\documentclass[a4paper]{amsart}

\usepackage{amsmath,amsthm,amssymb} 
\usepackage[a4paper,margin=3cm]{geometry}
\usepackage{hyperref}

\usepackage{enumitem}
\setlist[itemize]{itemsep=0ex,label=--}
\setlist[enumerate]{itemsep=1ex,topsep=1ex}

\newcommand\itv[5]{\left #1#2,#4\right#5}
\newcommand{\id}{\mathord{\operatorname{id}}} 
\newcommand{\red}{\mathrm{r}}
\newcommand{\full}{\mathrm{f}}
\newcommand{\Ll}{{\mathcal L}} 
 
\newcommand{\Kk}{{\mathcal K}} 
\newcommand{\Dd}{{\mathcal D}} 
\newcommand{\RR}{\mathbb{R}}
\newcommand{\NN}{\mathbb{N}} 
\newcommand{\ZZ}{\mathbb{Z}} 
\newcommand{\CC}{\mathbb{C}}
\newcommand{\FO}{\mathbb{F}O} 
\newcommand{\FU}{\mathbb{F}U} 
\newcommand{\Cst}{$C^*$-\relax}
\newcommand{\qdim}{\mathop{\operatorname{qdim}}}
\renewcommand{\dim}{\mathop{\operatorname{dim}}}
\newcommand{\Tr}{\mathop{\operatorname{Tr}}}
\newcommand{\Span}{\mathop{\operatorname{Span}}}
\newcommand{\Corep}{\mathop{\operatorname{Corep}}}
\newcommand{\Hom}{\mathop{\operatorname{Hom}}}
\newcommand{\Img}{\mathop{\operatorname{Im}}}
\newcommand{\Ker}{\mathop{\operatorname{Ker}}}
\newcommand{\Irr}{\mathop{\operatorname{Irr}}}
\newcommand{\rk}{\mathop{\operatorname{rank}}}

\renewcommand{\Re}{\mathop{\operatorname{Re}}}
\renewcommand{\Im}{\mathop{\operatorname{Im}}}
\renewcommand{\d}{\mathrm{d}}
\newcommand{\ts}{\textstyle}
\renewcommand{\ss}{\scriptscriptstyle}
\newcommand{\ds}{\displaystyle}
\newcommand{\jok}{\bigstar}

\theoremstyle{plain} 
\newtheorem{theorem}{Theorem}[section]
\newtheorem{lemma}[theorem]{Lemma}
\newtheorem{proposition}[theorem]{Proposition}
\newtheorem{corollary}[theorem]{Corollary} 
\theoremstyle{definition} 
\newtheorem{definition}[theorem]{Definition}
 
\newtheorem{remark}[theorem]{Remark}

\newtheorem*{maintheorem}{Theorem}

\title[Orthogonal free quantum group factors are strongly
$1$-bounded]{Orthogonal free quantum group factors are \\ strongly $1$-bounded}
\author{Michael Brannan \and Roland Vergnioux}

\begin{document}


\begin{abstract}
  We prove that the orthogonal free quantum group factors $\Ll(\FO_N)$ are
  strongly \mbox{$1$-bounded} in the sense of Jung.  In particular, they are not
  isomorphic to free group factors.  This result is obtained by
  establishing a spectral regularity result for the edge reversing operator on
  the quantum Cayley tree associated to $\FO_N$, and combining this result with
  a recent free entropy dimension rank theorem of Jung and Shlyakhtenko.
\end{abstract}

\maketitle

\section{Introduction}

The theory of discrete quantum groups provides a rich source of interesting
examples of C$^\ast$-algebras and von Neumann algebras.  In addition to ordinary
discrete groups, there is a wealth of examples and phenomena arising from
genuinely quantum groups \cite{1504.00092, Woronowicz_Pseudogroups, MR2554941,
  MR3452270, MR3478874, 1410.6238}.  Within the class of non-amenable discrete
quantum groups, the so-called {\it free quantum groups} of Wang and Van Daele
\cite{VanDaeleWang_Universal, Wang_FreeProd} somehow form the most prominent
examples.

In this paper, our main focus is on the structural theory of a family of
II$_1$-factors associated to a special family of free quantum groups, called the
{\it orthogonal free quantum groups}.  Given an integer $N \ge 2$, the
orthogonal free quantum group $\FO_N$ is the discrete quantum group defined via
the full Woronowicz C$^\ast$-algebra
\begin{displaymath}
  C^*_\full(\FO_N) = \langle u_{ij}, 1\leq i,j\leq N \mid u =
  [u_{ij}]~\text{unitary}, \ u_{ij} = u_{ij}^* \ \forall i,j \rangle.
\end{displaymath}
The C$^{\ast}$-algebra $C^*_\full(\FO_N)$ can be interpreted simultaneously
as a free analogue of the C$^\ast$-algebra of continuous functions on the real
orthogonal group $O_N$, and also as a ``matricial'' analogue of the full free
group C$^\ast$-algebra $C^*_\full((\ZZ/2\ZZ)^{\star N})$. Indeed, by
quotienting by the commutator ideal or by setting $u_{ij} = 0$ ($i \ne j$),
respectively, we obtain surjective Woronowicz-C$^\ast$-morphisms
\begin{displaymath}
  C^*_\full(\FO_N) \to C(O_N), \qquad C^*_\full(\FO_N) \to
  C^*_\full((\ZZ/2\ZZ)^{\star N}).
\end{displaymath}

Using the (tracial) Haar state $h : C^*_\full(\FO_N) \to \CC$, the GNS
construction yields in the usual way a Hilbert space $\ell^2(\FO_N)$ and a
corresponding von Neumann algebra
$\Ll(\FO_N) = \pi_h(C^*_\full(\FO_N))'' \subseteq B(\ell^2(\FO_N))$, where
$\pi_h$ denotes the GNS representation.  Over the past two decades, the
structure of the algebras $\Ll(\FO_N)$ has been investigated by many hands, and
in many respects $\FO_N$ and $\Ll(\FO_N)$ ($N \ge 3$) were shown to share many
properties with free groups $F_n$ and their von Neumann algebras $\Ll(F_n)$.  

For example, $\Ll(\FO_N)$ is a full type II$_1$-factor, it is strongly solid,
and in particular prime and has no Cartan subalgebra; it has the Haagerup
property (HAP), is weakly amenable with Cowling-Haagerup constant 1 (CMAP), and
satisfies the Connes' Embedding conjecture \cite{Banica_ReprOrtho, MR2355067,
  MR3391904, MR2995437, MR3084500, BrannanCollinsVergnioux, MR3455859}.
Moreover, it is known that $\Ll(\FO_N)$ behaves asymptotically like a free group
factor in the sense that the canonical generators of $\Ll(\FO_N)$ become
strongly asymptotically free semicircular systems as $N \to \infty$
\cite{MR2341011, MR3270793}.

With these many similarities between $\Ll(\FO_N)$ and $\Ll(F_n)$ at hand, the
following question naturally arises: \\ \\
{\it Can $\Ll(\FO_N)$ be isomorphic to a
  free group factor}?   \\ 

This particular question has been circulating within the operator algebra and quantum group communities ever since the publication of Banica's thesis \cite{Banica_ReprOrtho, MR1484551} in the mid 1990's, which first connected the corepresentation theory of free quantum groups to Voiculescu's free probability theory.  This deep connection with free independence established  by Banica was a direct inspiration for the many structural results for $\Ll(\FO_N)$ described in the previous paragraph.  In this paper, our main objective is to finally answer the above question in the negative.

\bigskip

The first evidence suggesting a negative answer to an isomorphism with a free
group factor came from the work of the second author \cite{Vergnioux_Paths},
where the $L^2$-cohomology of $\FO_N$ was investigated.  There it was shown that
the first $L^2$-Betti number of $\FO_N$ vanishes for all $N \ge 3$, see also
\cite{1701.06447}.  Combining this result with some deep work of
Connes-Shlyakhtenko \cite{MR2180603}, Jung \cite{MR2015736}, and
Biane-Capitaine-Guionnet \cite{BianeCapitaineGuionnet} on free entropy
dimension, it was shown by Collins and the authors
\cite{BrannanCollinsVergnioux} that
\begin{align}\label{BCV-result}
  \delta_0(u) = \delta^*(u) = 1 \qquad (N \ge 4)\footnotemark,
\end{align}\footnotetext{These values are also conjectured to hold for $N=3$ but in that case Connes embeddability of $\Ll(\FO_N)$ is open and therefore only the inequality $-\infty \le \delta_0(u) \le \delta^*(u) \le 1$ is known.}where
$u = (u_{ij})_{1 \le i,j \le N}$ is the set of canonical self-adjoint generators
of $\Ll(\FO_N)$, and $\delta_0, \delta^*$ are Voiculescu's (modified)
microstates free entropy dimension and non-microstates free entropy dimension,
respectively \cite{MR1887698, MR1371236, Voiculescu_V}.

Recall that if $X$ is a finite set of self-adjoint generators of a finite von
Neumann algebra $M$ with faithful normal tracial state $\tau$, $\delta_0(X)$ can
be interpreted as an asymptotic Minkowski dimension of the space of microstates
of $X$.  The fundamental problem relating to $\delta_0$ is whether or not it is
a W$^\ast$-invariant: If $X, X' \subset M_{sa}$ are finite sets generating the
same von Neumann subalgebra, do we have $\delta_0(X) = \delta_0(X')$?  If the
answer to this question is yes, then this would solve the well-known free group
factor isomorphism problem since $\Ll(F_n)$ admits a finite generating set $X$
with $\delta_0(X) = n$ \cite{MR1887698}.

In the remarkable work \cite{Jung_OneBounded}, Jung introduced a certain
technical strengthening of the condition $\delta_0(X) \le \alpha$ (see Section
\ref{FED} for details), which he called {\it $\alpha$-boundedness} of $X$.
There, Jung proved the remarkable result that if $(M, \tau)$ is a finite von
Neumann algebra generated by a $1$-bounded set $X\subset M_{sa}$ containing at
least one element with finite free entropy, then every other self-adjoint
generating set $X'$ of $M$ has $\delta_0(X') \le 1$.  In this case, we call $M$
a {\it strongly $1$-bounded von Neumann algebra}, and $\delta_0$ becomes a
$W^*$-invariant for $M$.  Note, in particular, that any strongly $1$-bounded von
Neumann algebra cannot be isomorphic to any (interpolated) free group factor
$\Ll(F_r)$ $(r \ge 2)$ \cite[Corollary 3.6]{Jung_OneBounded}.

The main result of this paper is an upgrade of the free entropy dimension
estimate \eqref{BCV-result} to the following theorem:

\begin{maintheorem}[See Theorem \ref{s1bdd} and Corollary \ref{consequence}]
  For each $N \ge 3$, $\Ll(\FO_N)$ is a strongly $1$-bounded von Neumann
  algebra.  In particular, $\Ll(\FO_N)$ is never isomorphic to an interpolated
  free group factor.
\end{maintheorem}

\bigskip

Note that II$_1$-factors which have property Gamma, or have a Cartan subalgebra,
or are tensor products of infinite dimensional factors, are automatically
strongly $1$-bounded by \cite{Jung_OneBounded}. This is not the case of
$\Ll(\FO_N)$.  Instead, our proof of strong $1$-boundedness relies on and is
heavily inspired by recent works of Jung \cite{1602.04726} and Shlyakhtenko
\cite{Shlyakhtenko_OneBounded}.  

If $F$ is an $l$-tuple of non-commutative polynomials over $m$ variables, one
can compute Voiculescu's free derivative $\partial F$ which yields by evaluation
an operator $\partial F(X)\in M \otimes M^{op} \otimes B(\CC^m,\CC^l)$.  In
\cite{1602.04726}, Jung showed that if $(M,\tau)$ is a finite von Neumann
algebra, $X \in M_{sa}^m$ is an $m$-tuple satisfying the polynomial relations
$F(X) = 0$, then $X$ is $\alpha$-bounded with $\alpha = m-\rk(\partial F(X))$,
{\it provided} that $\partial F(X)^*\partial F(X)$ has a non-zero modified
L\"uck-Fuglede-Kadison determinant. See Section \ref{sec_notation} for any
undefined notation and terms here.

In \cite{Shlyakhtenko_OneBounded}, Shlyakhtenko gave another proof of Jung's
result above using non-microstates free entropy techniques, and moreover used
this result to show that whenever $\Gamma$ is an infinite, finitely generated
and finitely presented sofic group with vanishing first $L^2$-Betti number, then
$\Ll(\Gamma)$ is strongly $1$-bounded.  The key idea here being that there
always exists a canonical system of generators
$X \in \mathbb Q[\Gamma]_{sa}^{m} \subset \Ll(\Gamma)^m_{sa}$ and
rational-polynomial relations $F(X)=0$, where
\begin{enumerate}
\item \label{con1} $m-\rk(\partial F(X)) = \beta^{(2)}_1(\Gamma) - \beta^{(2)}_0(\Gamma)+1$.
\item \label{con2} $\partial F(X)^*\partial F(X)$ has a non-zero modified
  L\"uck-Fuglede-Kadison determinant.
\end{enumerate}   
Note that the first condition above holds for any finitely generated finitely
presented group, whereas the second, typically very difficult to check condition
comes for ``free'' for sofic groups -- thanks to Elek and Szab\'o's solution to
L\"uck's determinant conjecture for sofic groups \cite{MR2178069}.

Returning to the quantum groups $\FO_N$, it is very natural to view these
objects as quantum analogues of finitely generated, finitely presented sofic
groups with vanishing first $L^2$-Betti number. Indeed, $\FO_N$ is hyperlinear
in the sense of \cite{BrannanCollinsVergnioux}, and even residually finite in
the sense that the underlying Hopf $\ast$-algebra $\mathbb C[\FO_N]$ is
residually finite-dimensional \cite{MR3336732}. However discrete quantum groups
are much more linear in nature than ordinary discrete groups and it is not clear
whether there is a quantum analogue of soficity that would allow one to prove
L\"uck's determinant conjecture for discrete quantum group rings.

\bigskip

Our strategy in this paper for proving our strong $1$-boundedness theorem, which
now can be seen as a quantum analogue of Shlyakhtenko's sofic group result, is
to first take the canonical system of generators
$X = u = (u_{ij})_{1 \le i,j \le N}$ and form the natural vector of quadratic
relations $F(X) = 0$ associated to the defining orthogonality relations of
$\FO_N$.  We then proceed to show conditions \eqref{con1} and \eqref{con2} from
above for this choice of $F$ and $X$.  Establishing \eqref{con1} turns out to be
a relatively straightforward adaptation of the results in the group case (see
Lemma \ref{lem_rank}).

On the other hand, establishing \eqref{con2} directly turns out
to be much more involved and constitutes the main technical component of the
paper.  Without the analogue of Elek-Szab\'o's results in this setting, we must
check the determinant condition for $D = \partial F(X)^*\partial F(X)$
explicitly.  This amounts to proving the integrability of the function
$\log_+:[0,\infty) \to \mathbb R$ with respect to the spectral measure of $D$,
where $\log_+(t) = \log (t)$ if $t > 0$ and $\log_+(0)= 0$.  

This integrability condition is established by proving an identification of $D$,
up to amplification and unitary equivalence, with the operator
$2(1 + \text{Re}(\Theta))$, where $\Theta \in B(K)$ is the so-called {\it
  edge-reversing operator} of the quantum Cayley tree associated to the quantum
group $\FO_N$.  Here, $K$ denotes the {\it edge Hilbert space} associated to the quantum Cayley tree.   Quantum Cayley graphs where introduced by the second author in
\cite{Vergnioux_Cayley} and studied further in \cite{Vergnioux_Paths}, where
they were a key ingredient to prove the vanishing of the first $L^2$-Betti
number of $\FO_N$. More specifically, a large part of
\cite{Vergnioux_Cayley,Vergnioux_Paths} was devoted to the study of the
eigenspaces $K_g^{\pm} = \Ker(\Theta\pm\id)$.

In the quantum case, $\Theta$ is not involutive and the understanding of its
behavior on the orthogonal complement of $K_g^+\oplus K_g^-$ is essential for the study of
the integrability condition of $D$. In the present article, we unveil a shift
structure for the action of $\Re(\Theta)$ on the orthogonal complement of $K_g^+\oplus K_g^-$,
reducing the initial problem to an integrability question for real parts of
weighted shifts.

\bigskip

Finally, let us conclude this introduction with the following natural question:
Although we now know that $\Ll(\FO_N)$ is not isomorphic to a free group factor,
could it still be possible that $\Ll(\FO_N)$ is isomorphic to $\Ll(\Gamma)$ for
some other classical discrete group $\Gamma$?  In particular, what about
$\Gamma$ being an ICC lattice in $SL(2,\CC)$?  For such $\Gamma$, it is known
that $\Ll(\Gamma)$ is a full, strongly solid, strongly $1$-bounded II$_1$-factor
which has the HAP and the CMAP.  Note also that \cite{Houdayer_StronglySolid}
provides other examples of groups $\Gamma$ such that $\Ll(\Gamma)$ satisfies the
same properties.


\bigskip

The remainder of the paper is organized as follows. In
Section~\ref{sec_notation} we introduce some basic notation and preliminaries
about discrete quantum groups and free entropy dimension. In
Section~\ref{sec_reversing} we proceed to the spectral analysis of the reversing
operator, reducing the determinant class question to the case of weighted
shifts. In Section~\ref{sec_free_entropy} we study the relations in $\FO_N$ from
the point of view of free entropy dimension and we prove the main
$1$-boundedness result. Finally the Appendix summarizes some background results
from \cite{Vergnioux_Cayley,Vergnioux_Paths} on quantum Cayley graphs used in Section~\ref{sec_reversing}.

\subsection*{Acknowledgments} The authors thank \'Eric Ricard and Dimitri
Shlyakhtenko for valuable conversations and encouragement.  The authors also thank the anonymous referee for thoroughly reading the manuscript and providing valuable comments and suggestions.  This research was partially supported  by NSF grant DMS-1700267.  

\section{Notation and Preliminaries}
\label{sec_notation}

Scalar products are linear on the right. We denote by $\otimes$ the tensor
product of Hilbert spaces and the minimal tensor product of \Cst algebras. We
use the leg numbering notation for elements of multiple tensor products.   The
flip operator on Hilbert spaces is denoted $\Sigma : H\otimes K \to K\otimes H$.  For example, if $H, K, L$ are Hilbert spaces, $T \in B(H \otimes K)$, $S \in B(K)$, then $T_{12} \in B(H \otimes K \otimes L)$, $S_2 \in B(H \otimes K \otimes L)$, $T_{32} \in B(L \otimes K \otimes H)$ are given by $T \otimes \id$, $\id \otimes S \otimes \id$, and $(\id \otimes \Sigma)(\id \otimes T)(\id \otimes \Sigma)$, respectively. 

Let us denote $\log_+(t) = \log(t)$ for $t>0$ and $\log_+(0)=0$. The function
$\log_+$ can be applied to positive operators using Borel functional
calculus. If $M$ is a von Neumann algebra with finite faithful normal trace
$\tau$, L\"uck's modified Fuglede-Kadison determinant of $x\in M$ is $\Delta^+_\tau(x) = \exp(\tau(\log_+(|x|))) \in [0,\infty)$. We will
say that $x\in (M,\tau)$ is of {\it determinant class} if $\Delta^+_\tau(x) >0$,
i.e. $\tau(\log_+(|x|)) > -\infty$.  Here, the quantity $\tau(\log_+(|x|))$ is computed via the Lebesgue integral
\begin{displaymath}
  \tau(\log_+(|x|)) = 
  \int_{(0,\infty)} \log(\lambda)\d\mu(\lambda) = 
  \lim_{\epsilon \to 0^+} \int_\epsilon^{\|x\|} 
  \log(\lambda)\d\mu(\lambda) \in [-\infty,\infty[  
\end{displaymath}
where $\mu$ denotes the spectral distribution of $|x|$ induced by $\tau$.

We denote $L^2(M,\tau)$ the GNS space, equipped with the natural left and right
$M$-module structures $x\hat y = \widehat{xy}$ and
$\hat yx = \widehat{yx} = Jx^*J\hat y$, where $\hat x$ denotes the image in
$L^2(M,\tau)$ of $x\in M$. We denote $M^\circ$ the opposite von Neumann algebra,
$L^2(M^\circ,\tau)$ the corresponding GNS space with left and right actions of
$M^\circ$ denoted $\hat y x = \widehat{yx}$, $x\hat y = \widehat{xy}$, where we
use the product of $M$.

\subsection{Discrete quantum groups}

We use the setting of {\em Woronowicz} \Cst algebras
\cite{Woronowicz_LesHouches}, i.e. unital \Cst algebras $A$ equipped with a
$*$-homomorphism $\Delta : A \to A\otimes A$ such that
$(\Delta\otimes\id)\Delta = (\id\otimes\Delta)\Delta$ and
$\Delta(A)(1\otimes A)$, $\Delta(A)(A\otimes 1)$ are dense in $A\otimes A$.
Woronowicz proved the existence and uniqueness of a state $h \in A^*$ such that
$(h\otimes\id)\Delta = (\id\otimes h)\Delta = h(\cdot)1$, called the Haar state
\cite{Woronowicz_Pseudogroups}. The Woronowicz \Cst algebra $(A,\Delta)$ is
called {\em reduced} if the GNS representation $\pi_h$ associated with $h$ is
faithful. Note that $(\pi_h\otimes\pi_h)\Delta$ factors through $\pi_h$
and in this way $\pi_h(A)$ is naturally a reduced Woronowicz \Cst algebra.

If $\Gamma$ is a discrete group, the full and reduced \Cst algebras
$C^*_\full(\Gamma)$, $C^*_\red(\Gamma)$ are Woronowicz \Cst algebras with
respect to the coproducts given by $\Delta(g) = g\otimes g$, where group
elements $g\in\Gamma$ are identified with the corresponding unitary elements in
$C^*_\full(\Gamma)$, $C^*_\red(\Gamma)$. In general we shall interpret
Woronowicz \Cst algebras as {\em discrete quantum group} \Cst algebras and
denote $(A,\Delta) = (C^*(\Gamma),\Delta)$, where $\Gamma$ is the discrete
quantum group associated with $(A,\Delta)$. There is always a reduced version
$C^*_\red(\Gamma)$ of $C^*(\Gamma)$, as above, as well as a full version
$C^*_\full(\Gamma)$. The von Neumann algebra of $\Gamma$ is
$\Ll(\Gamma) = C^*_\red(\Gamma)'' \subset B(\ell^2(\Gamma))$, where
$\ell^2(\Gamma)$ is the GNS space of $h$.

The main class of examples for the present article are the orthogonal free
quantum groups $\FO(Q)$
\cite{Wang_FreeProd,VanDaeleWang_Universal,Banica_ReprOrtho}, where
$Q \in GL_N(\CC)$ is a matrix such that $Q\bar Q \in \CC I_N$. The corresponding full Woronowicz C$^\ast$-algebras are defined by
generators and relations:
\begin{displaymath}
  C^*_\full(\FO(Q)) = \langle u_{ij}, 1\leq i,j\leq n \mid u~\text{unitary}, 
  Q\bar uQ^{-1} = u\rangle
\end{displaymath}
where $\bar u = (u_{ij}^*)_{ij}$, with the coproduct given on generators by
$\Delta(u_{ij}) = \sum_k u_{ik}\otimes u_{kj}$. In the case $Q = I_N$ we denote
$\FO(Q) = \FO_N$.

\begin{remark}
  In the literature, another commonly used (dual) notation for the
  C$^\ast$-algebra $C^*_\full(\FO_N)$ is $C^u(O_N^+)$, or sometimes $A_o(N)$.
  The notation $C^u(O_N^+)$ refers to the fact that this C$^\ast$-algebra can be
  viewed as a free analogue of the C$^\ast$-algebra of continuous functions on
  the real orthogonal group $O_N$.  In terms of Pontryagin duality for quantum
  groups, $O_N^+ = \widehat{\FO_N}$ is the compact dual of the discrete quantum
  group $\FO_N$ and the Fourier transform \cite{MR1059324} induces the
  identifications $C^u(O_N^+) = C^*_\full(\FO_N)$ and
  $L^\infty(O_N^+) =\pi_h(C^u(O_N^+))'' = \Ll(\FO_N)$.  Since our perspective is
  to view our objects as quantum analogues of discrete groups, we stick to the
  notation $\FO_N$.
\end{remark}

Denote by $\pi_h : C^*(\Gamma) \to B(\ell^2(\Gamma))$ the GNS representation
associated with the Haar state, with canonical cyclic vector
$\xi_0 \in \ell^2(\Gamma)$. The multiplicative unitary \cite{BaajSkandalis} of
$\Gamma$ is the unitary operator $V$ acting on
$\ell^2(\Gamma)\otimes \ell^2(\Gamma)$ and given by the formula
$V(x\xi_0\otimes y\xi_0) = \Delta(x)(1\otimes y)(\xi_0\otimes\xi_0)$ for $x$,
$y \in C^*(\Gamma)$. It satisfies the so-called pentagonal equation
$V_{12}V_{13}V_{23} = V_{23}V_{12}$. The reduced algebra
$C^*_\red(\Gamma) \subset B(\ell^2(\Gamma))$ can be recovered as the closed
linear span of the slices $(\varphi\otimes\id)(V)$, $\varphi \in B( \ell^2(\Gamma))_*$, with
coproduct $\Delta(x) = V(x\otimes 1)V^*$. Another useful operator is the polar
part of the antipode.  This is the involutive unitary $U \in B(\ell^2(\Gamma))$ given by $U(x\xi_0) = R(x)\xi_0$ ($x \in C^*(\Gamma)$), where $R:C^*(\Gamma) \to C^*(\Gamma)^{\circ}$ is the unitary antipode.

The {\em dual algebra} $c_0(\Gamma)$ can be defined as the closed linear span,
in $B(\ell^2(\Gamma))$, of the slices $(\id\otimes\varphi)(V)$,
$\varphi \in B( \ell^2(\Gamma))_*$, and equipped with the coproduct
$\Delta : c_0(\Gamma) \to M(c_0(\Gamma) \otimes c_0(\Gamma)), a \mapsto
V^*(1\otimes a) V$
(following \cite{BaajSkandalis}). It is a (not necessarily unital) Hopf-\Cst
algebra \cite{Vallin}. We have then
$V \in M(c_0(\Gamma)\otimes C^*_\red(\Gamma))$. We denote
$p_0 = (\id\otimes h)(V) \in c_0(\Gamma)$, which is also the orthogonal
projection onto $\CC\xi_0 \subset H$.

In the ``classical
case'', when $\Gamma$ is a real discrete group, on can check that
$V = \sum_{g\in\Gamma} \delta_g\otimes\pi_h(g)$, where
$\delta_g$ is the characteristic function of $\{g\}$ acting by
pointwise multiplication on $\ell^2(\Gamma)$ and $\pi_h(g)$ is the
operator of left translation by $g$. In particular $c_0(\Gamma)$ identifies
with the \Cst algebra of functions on $\Gamma$ vanishing at infinity, as the
notation suggests.

\subsection{Quantum Cayley graphs}

Let $p_1\in Z(M(c_0(\Gamma)))$ be a central projection such that $Up_1 = p_1U$
and $p_0p_1 = 0$.  The {\em quantum Cayley graph} $X$ \cite{Vergnioux_Cayley}
associated to $(\Gamma, p_1)$ is given by
\begin{itemize}
\item the vertex and edge Hilbert spaces $\ell^2(X^{(0)}) = \ell^2(\Gamma)$ and
  $\ell^2(X^{(1)}) = \ell^2(\Gamma)\otimes p_1\ell^2(\Gamma)$,
\item the vertex and edge \Cst algebras $c_0(X^{(0)}) = c_0(\Gamma)$ and
  $c_0(X^{(1)}) = c_0(\Gamma)\otimes p_1c_0(\Gamma)$, naturally represented on
  the corresponding Hilbert spaces,
\item the reversing operator
  $\Theta = \Sigma(1\otimes U)V(U\otimes U)\Sigma \in B(\ell^2(X^{(1)}))$,
\item the boundary operator
  $E = V \in B(\ell^2(X^{(1)}), \ell^2(X^{(0)})\otimes \ell^2(X^{(0)}))$.
\end{itemize}
For brevity we denote $\ell^2(X^{(0)}) = \ell^2(\Gamma) = H$ and
$\ell^2(X^{(1)}) = H\otimes p_1H = K$. The fact that $p_1$ commutes with $U$
ensures that $K$ is stable under $\Theta$ and $\Theta^*$. Using the densely
defined ``augmentation form'' $\epsilon : H \to \CC$ induced by the co-unit of $C^*_{\text{f}}(\Gamma)$, one
can also consider source and target maps $E_1 = (\id\otimes\epsilon) E$,
$E_2 : (\epsilon\otimes\id)E : K \to H$. When $p_1$ has finite rank, these are
in fact bounded operators.

In the classical case, $p_1$ is the characteristic function of a subset
$S\subset \Gamma$ such that $S^{-1} = S$ and $e\notin S$. Denoting by
$(e_g)_{g\in\Gamma}$ the canonical Hilbertian basis of $\ell^2(\Gamma)$, it is
easy to compute $\Theta(e_g\otimes e_h) = e_{gh}\otimes e_{h^{-1}}$ and
$E(e_g\otimes e_h) = e_g\otimes e_{gh}$. Hence the operators $\Theta$, $E$
encode the graph structure of the usual Cayley graph associated with
$(\Gamma,S)$, with edges given by ``source, direction'' pairs
$(g,h) \in \Gamma\times S$. Note that in the quantum case, $\Theta$ is always
unitary, but not necessarily involutive. More details about quantum Cayley
graphs, especially in the case of trees, are given in the Appendix.

\bigskip

If $\Gamma$ is a discrete group, the unitaries $v = g \in C^*_\full(\Gamma)$ or
$C^*_\red(\Gamma)$ corresponding to group elements can be recovered as those
unitaries $v$ which are {\em group-like}, i.e. satisfy the relation
$\Delta(v) = v\otimes v$. More generally, a {\em unitary corepresentation} of a
Woronowicz \Cst algebra $C^*(\Gamma)$ on a Hilbert space $H$ is a unitary
element $v \in M(K(H)\otimes C^*(\Gamma)))$ such that
$(\id\otimes\Delta)(v) = v_{12} v_{13} \in M(K(H)\otimes C^*(\Gamma)\otimes
C^*(\Gamma))$.  Here, $K(H)$ denotes the C$^\ast$-algebra of compact operators on the Hilbert space $H$.
Applying $\id\otimes\pi_h$ yields a bijection between corepresentations of
$C^*(\Gamma)$ and $C^*_\red(\Gamma)$, hence one can speak of corepresentations
of the discrete quantum group $\Gamma$. 

We denote by $\Corep(\Gamma)$ the category of finite dimensional
corepresentations of $\Gamma$. It is a rigid tensor \Cst category, with direct
sum $v\oplus w$, and tensor product $v\otimes w = v_{13}w_{23}$. The space of
$v \in \Corep(\Gamma)$ is denoted $H_v$ and we put $\dim v = \dim H_v$. We write
$v\subset w$ (resp. $v\simeq w$) if $\Hom(v,w)$ contains an injective
(resp. bijective) map, and we choose a set $\Irr(\Gamma)$ of representatives of
irreducible corepresentations up to equivalence. Any corepresentation dual to
$v$ will be denoted $\bar v$, and the quantum (or intrinsic) dimension of $v$ is
denoted $\qdim v$.  See e.g. \cite{MR3204665} for more details. 

The structure of $c_0(\Gamma)$ can be described using the theory of
corepresentations. More precisely, there is a canonical dense subspace of $H$
that can be identified with $\bigoplus_{\alpha\in\Irr\Gamma} B(H_\alpha)$ in
such a way that $c_0(\Gamma) \subset B(H)$ identifies with
$c_0{-}\bigoplus_{\alpha\in\Irr\Gamma} B(H_\alpha)$ acting on the dense subspace
by left multiplication. Moreover this gives a decomposition of the
multiplicative unitary $V$ (which is also a unitary corepresentation):
$V = \sum_{\alpha\in\Irr\Gamma} \alpha \in M(c_0(\Gamma)\otimes
C^*_\red(\Gamma))$.
We denote $p_\alpha \in c_0(\Gamma) \subset B(H)$ the minimal central projection
corresponding to the block $B(H_\alpha)$, so that $H = \bigoplus p_\alpha H$ and
$p_\alpha H\simeq B(H_\alpha)$. For the trivial corepresentation
$\tau = \id_\CC\otimes 1$ we have $p_\tau = p_0$.

\subsection{Free entropy dimension} \label{FED}

There are two main approaches to free entropy dimension, based respectively on
microstates and conjugate variables. The tools that we are going to use in this
article are more closely related to the second one, although the invariance of
strong $1$-boundedness under von Neumann algebra isomorphisms is proved by Jung in the
first framework.

\bigskip

For a tuple of indeterminates $x = (x_1, \ldots, x_m)$, we denote
$\CC\langle x\rangle$ the corresponding algebra of noncommutative
polynomials. The free difference quotient $\partial_i$ is the unique derivation
$\partial_i : \CC\langle x\rangle \to \CC\langle x\rangle\otimes \CC\langle
x\rangle$
such that $\partial_i x_j = \delta_{ij}(1\otimes 1)$, where
$\CC\langle x\rangle\otimes \CC\langle x\rangle$ is equipped with the bimodule
structure $P\cdot(R\otimes S)\cdot Q = PR\otimes SQ$. We denote
$\partial P = \sum \partial_i P \otimes e_i^* \in \CC\langle x\rangle\otimes
\CC\langle x\rangle \otimes (\CC^m)^*$
and, if $P = (P_1,\ldots,P_l) \in \CC\langle x\rangle^l$,
$\partial P = \sum \partial_i P_j \otimes e_j \otimes e_i^* \in \CC\langle
x\rangle\otimes \CC\langle x\rangle \otimes B(\CC^m,\CC^l)$.

\bigskip

Fix a tuple $X = (X_1,\ldots, X_m)$ of self-adjoint elements in a von Neumann
algebra $M$ with faithful finite normal trace $\tau$, and denote
$W^*(X) \subset M$ the von Neumann subalgebra generated by $X$. We say that
$\xi_i \in L^2(M,\tau)$ is the (necessarily unique) conjugate variable of $X_i$
if $\xi_i \in L^2(W^*(X),\tau)$ and
$\langle \xi_i, P(X)\rangle = (\tau\otimes\tau) ((\partial_i P)(X))$ for all
$P \in \CC\langle x\rangle$. The free Fisher information of $X$ is
$\Phi^*(X) = \sum_i \|\xi_i\|_2^2$ if all conjugate variables exist, and $+\infty$
otherwise.

Replacing $M$ by a free product if necessary, one can assume that $M$ contains a
free family $S = (S_1,\ldots,S_m)$ of elements with $(0,1)$-semicircular law
with respect to $\tau$, which is also freely independent from $X$. The
non-microstates free entropy \cite{Voiculescu_V} is defined by
\begin{displaymath}
  \chi^*(X) = {\ts\frac 12} \int_0^{+\infty} \left({\ts\frac m{1+t}} - 
    \Phi^*(X+\sqrt t S)\right)\d t + {\ts\frac m2}\log(2\pi e)
  \in\itv[{-\infty},{+\infty}[,
\end{displaymath}
and the non-microstates free entropy dimension is 
\begin{displaymath}
  \delta^*(X) = m - \liminf_{\epsilon \to 0} 
  \frac{\chi^*(X+\sqrt\epsilon S)}{\log\sqrt\epsilon}.
\end{displaymath}
The (modified) microstates free entropy dimension $\delta_0(X)$ is defined by
the very same formula, using the relative microstates free entropy
$\chi(X+\sqrt\epsilon S:S)$ instead of $\chi^*(X+\sqrt\epsilon S)$ \cite{MR1371236}.

\bigskip

One can observe that we have $\delta_0(X) \leq \alpha$ {\bf iff}
$\chi(X+\sqrt\epsilon S:S) \leq (\alpha-m)|\log\sqrt \epsilon| +
o(\log\sqrt\epsilon)$
as $\epsilon \to 0$. Following Jung \cite{Jung_OneBounded}, one says that $X$ is
{\it $\alpha$-bounded} (for $\delta_0$) if it satisfies the slightly stronger
condition $\chi(X+\sqrt\epsilon S:S) \leq (\alpha-m)|\log\sqrt \epsilon| + K$
for small $\epsilon>0$ and some $K$ independent of $\epsilon$. Similarly, one
can say that $X$ is $\alpha$-bounded for $\delta^*$ if
$\chi^*(X+\sqrt\epsilon S) \leq (\alpha-m)|\log\sqrt \epsilon| + K$.

Recall that it is a major open question in free probability theory to decide whether
$\delta_0(X)$ is an invariant of $W^*(X)$.  Indeed, $\Ll(F_m)$ admits a tuple
of generators $X$ such that $\delta_0(X) = m$ \cite{MR1887698}, and therefore the W$^\ast$-isomorphism invariance of $\delta_0$ would provide a solution to the celebrated free group factor isomorphism problem.  Jung proved the following very
strong result: if $X$ is $1$-bounded and $\chi(X_i)>-\infty$ for at least one
$i$, then any other tuple $X'$ of self-adjoint generators of $W^*(X)$ is
$1$-bounded \cite{Jung_OneBounded}. In particular, in that case one cannot have
$W^*(X) \simeq \Ll(F_m)$ for $m\geq 2$. Let us also record the following deep
result comparing the two versions of free entropy: we always have
$\chi(X) \leq \chi^*(X)$ \cite{BianeCapitaineGuionnet}. In particular
$\chi(X+\sqrt\epsilon S:S) \leq \chi^*(X+\sqrt\epsilon S)$ so that
$1$-boundedness for $\delta^*$ implies $1$-boundedness for $\delta_0$.

\bigskip

Our main tool in this article is the following result, originally proved by Jung
in the microstates framework \cite{1602.04726}, and reproved by Shlyakhtenko
using non-microstates free entropy \cite{Shlyakhtenko_OneBounded}.  As above,
for any $P \in \CC\langle x\rangle^l$ and $X \in M^m_{sa}$ one can consider
$\partial P \in \CC\langle x\rangle \otimes \CC\langle x\rangle \otimes
B(\CC^m,\CC^l)$
and $\partial P(X) \in M \otimes M^\circ \otimes B(\CC^m,\CC^l)$ is a bounded
operator from $L^2(M,\tau)\otimes L^2(M^\circ,\tau)\otimes \CC^m$ to
$L^2(M,\tau)\otimes L^2(M^\circ,\tau)\otimes \CC^l$. The operator
$\partial P(X)$ moreover respects the right $M\otimes M^\circ$-module structures
given by
$(\zeta\otimes\xi\otimes\eta)\cdot (x\otimes y) = \zeta x\otimes y\xi \otimes
\eta$.
We denote by $\rk(\partial P(X))$ the Murray-von Neumann dimension over
$M\otimes M^\circ$ of the closure of $\Img (\partial P(X))$ in
$L^2(M,\tau)\otimes L^2(M^\circ,\tau)\otimes \CC^l$.

\begin{theorem} \label{thm_jung_shlyakhtenko} {\rm
    (\cite[Thm.~6.9]{Jung_OneBounded} and
    \cite[Thm.~2.5]{Shlyakhtenko_OneBounded})} Suppose that $X \in M_{sa}^m$
  satisfies the identity $F(X) = 0$ for $F \in \CC\langle x\rangle^l$. Assume
  moreover that $\partial F(X)$ is of determinant class. Then $X$ is
  $\alpha$-bounded for $\delta_0$ and $\delta^*$, with
  $\alpha = m-\rk(\partial F(X))$.
\end{theorem}

\section{Regularity of the reversing operator}
\label{sec_reversing}

In Section~\ref{sec_free_entropy} we will prove that $\Ll(\FO_N)$ is strongly
$1$-bounded by applying Theorem~\ref{thm_jung_shlyakhtenko} to the tuple $X$ of
canonical generators and a specific vector of relations $F$. It will turn out
that the real part of the operator $\partial F(X)$ is closely related to the
real part of the reversing operator $\Theta$ of the quantum Cayley graph of
$\FO_N$ with its canonical generators. In this section we prove the crucial
technical result that $1+\Re\Theta$ is of determinant class --- which is a
regularity property for the spectral measure of $\Re\Theta$ at the edge of the spectrum.  This result can be seen as further evidence that the quantum groups $\FO_N$ should be somehow regarded as quantum analogues of sofic or determinant class groups.

Note that all results in this section hold also in the non-Kac case, that is,
for all discrete quantum groups $\FO(Q)$ with $Q \in GL_N(\CC)$, $N\geq 2$,
$Q\bar Q\in\CC I_N$, except the ones isomorphic to the duals of $SU_{\pm 1}(2)$
--- which corresponds to the assumption $\qdim u > 2$.

\bigskip

Our study relies heavily on results about quantum Cayley graphs proved in
\cite{Vergnioux_Cayley,Vergnioux_Paths}, which we recall in the Appendix. Note
that the eigenspace $K_g^+ = \Ker(\Theta+\id)$ --- and, by symmetry
$K_g^{-}=\Ker(\Theta-\id)$ ---, were the main subject of study in
\cite{Vergnioux_Cayley,Vergnioux_Paths}. These stable subspaces behave trivially
with respect to the determinant class issue. Note also that in the classical
case, they span the whole of the ambient edge Hilbert space $K$, but not in the case of
$\FO_N$. Hence our main concern in the present article is the behavior of
$\Theta$ on $K_g^{+\bot}\cap K_g^{-\bot}$. 

Recall the definition~\ref{app_def_W} of the reflection operator $W$, which is
isometric and involutive. The study of $K_g^+$ in \cite{Vergnioux_Cayley} shows
that $W$ restricts to the identity on $K_g^+$ and $K_g^-$. More precisely, the
proof of \cite[Theorem~5.3]{Vergnioux_Cayley} shows that any vector
$\xi\in K_g^+$ can be written
$\xi = \zeta -(1+W)\eta + p_{\ss--}\Theta(\eta-\zeta)$ with $\zeta\in K_{\ss++}$
and $\eta\in K_{\ss+-}$, and $W$ restricts to the identity on $K_{\ss++}$ and
$K_{\ss--}$ by definition.
\begin{definition}
  We denote $K_s=\Ker(W-1)$, $K_a = \Ker(W+1)$ and
  $L = K_s\cap K_g^{+\bot}\cap K_g^{-\bot}$. We have then an orthogonal
  decomposition $K = K_g^+\oplus K_g^-\oplus K_a\oplus L$.
\end{definition}

The structure of $K_a$ and the behavior of $\Theta+\Theta^*$ on $K_a$ are quite
simple and we describe them in the next Proposition. We use the notation for the
left/right ascending/descending subspaces, e.g.  $K_{\ss+-} = p_{\ss+-}K$, which
is recalled in the Appendix.

\begin{proposition}
  We have $K_a \subset K_{\ss+-}\oplus K_{\ss-+}$ and the orthogonal projection
  onto $K_{\ss+-}$ restricts to an isomorphism $K_a\simeq K_{\ss+-}$ (up to a
  constant $\sqrt 2$). Moreover $K_a$ is $(\Theta+\Theta^*)$-stable and in the
  isomorphism with $K_{\ss+-}$ the operator $\Theta+\Theta^*$ corresponds to
  $-(r+r^*)$, where $r = -p_{\ss+-}\Theta p_{\ss+-}$.
\end{proposition}

\begin{proof}
  Since by definition $W$ restricts to the identity on $K_{\ss++}$ and
  $K_{\ss--}$ and switches $K_{\ss+-}$ and $K_{\ss-+}$ in an involutive and
  isometric way, the first two assertions are clear. The identity
  $W\Theta W=\Theta^*$ implies $[W,\Theta+\Theta^*]=0$ hence $K_a$ and $K_s$ are
  $(\Theta+\Theta^*)$-stable. Due to this stability and the inclusion
  $K_a \subset K_{\ss+-}\oplus K_{\ss-+}$ we have
  $\Theta+\Theta^* = (p_{\ss+-}+p_{\ss-+}) (\Theta+\Theta^*)
  (p_{\ss+-}+p_{\ss-+})$ on $K_a$.
  Since $p_{\ss+-}\Theta p_{\ss-+} = p_{\ss-+}\Theta p_{\ss+-} = 0$ by
  \ref{app_prop_orient} this yields
  \begin{displaymath}
    \Theta+\Theta^* = p_{\ss+-}(\Theta+\Theta^*)p_{\ss+-} + 
    p_{\ss-+}(\Theta+\Theta^*)p_{\ss-+} \text{~on~} K_a,
  \end{displaymath}
  and the last assertion follows.
\end{proof}

Note that the operator $r$ on $K_{\ss+-}$ was studied in
\cite{Vergnioux_Cayley}, and it is an infinite direct sum of right shifts with
explicit weights converging to $1$. Note however that we will be interested in
vector states corresponding to vectors in $K_{\ss++}$ whereas
$K_a \bot K_{\ss++}$, so that the behavior of $\Theta+\Theta^*$ on $K_a$ is not
relevant for our precise analytical issue.

Now we turn to the study of $\Theta+\Theta^*$ on $L$. It turns out that it also
behaves like the real part of a shift, but the study is slightly more
involved. Recall the shorthand notation $r = -p_{\ss+-}\Theta p_{\ss+-}$,
$s = p_{\ss+-}\Theta p_{\ss++}$ and $s' = p_{\ss+-}\Theta^*p_{\ss--}$.

\begin{proposition} 
  \label{prop_descr_L}
  Consider the map
  $\Lambda = (1+W)(r-r^*) + 2 (s^*-s^{\prime *}) : K_{\ss+-} \to K$. Then
  $\Lambda$ is injective, $\overline{\Img\Lambda} = L$ and
  $\Lambda^*\Lambda =8-2(r+r^*)^2$.
\end{proposition}

\begin{proof}
  We note that $s^* = p_{\ss++}\Theta^* p_{\ss+-}$ is injective on $K_{\ss+-}$:
  indeed the weights $s_{k,l}$ appearing in \ref{app_weights} vanish only for
  $l=0$, and $q_0 K_{\ss+-} = \{0\}$. In particular $p_{\ss++}\Lambda = 2s^*$ is
  injective, hence $\Lambda$ is injective.

  It is clear from the definitions that $L$ and $\overline{\Im\Lambda}$ are
  subspaces of $K_s$. Hence we have $\overline{\Im\Lambda} = L$ {\bf iff}
  $\Ker\Lambda^* \cap K_s = L^\bot\cap K_s = K_g^+\oplus K_g^-$. But we have
  $K_g^+\oplus K_g^- \subset K_s$ and
  $K_g^+\oplus K_g^- = \Ker(\Theta-\id) \oplus \Ker(\Theta+\id) =
  \Ker(\Theta^2-\id) = \Ker(\Theta-\Theta^*)$.
  Hence it suffices to prove that $\Lambda^*(\zeta) = 0$ $\Leftrightarrow$
  $\Theta\zeta = \Theta^*\zeta$ for $\zeta \in K_s$. The second identity is
  equivalent to the four equations obtained by applying $p_{\ss++}$,
  $p_{\ss+-}$, $p_{\ss-+}$ and $p_{\ss--}$. 

  Since $\zeta = W\zeta$, the equations
  $p_{\ss++}\Theta\zeta = p_{\ss++}\Theta^*\zeta$ and
  $p_{\ss--}\Theta\zeta = p_{\ss--}\Theta^*\zeta$ are trivial --- indeed we have
  e.g. for the first one:
  \begin{align*}
    p_{\ss++}\Theta\zeta &= p_{\ss++}p_{\ss\jok+}\Theta\zeta 
    = p_{\ss++}\Theta p_{\ss-+}\zeta + p_{\ss++}\Theta p_{\ss--}\zeta 
    \quad \text{by Proposition~\ref{app_prop_orient}} \\
    &= p_{\ss++}\Theta p_{\ss-+}W\zeta + p_{\ss++}\Theta p_{\ss--}\zeta \\
    &= p_{\ss++}\Theta^* p_{\ss+-}\zeta + p_{\ss++}\Theta^* p_{\ss--}\zeta 
    \quad \text{by Propositions~\ref{app_def_W} and~\ref{app_prop_special}} \\
      &= p_{\ss++}\Theta^* p_{\ss\jok-}\zeta = p_{\ss++}\Theta^* \zeta
    \quad \text{by Proposition~\ref{app_prop_orient}.}
  \end{align*}
  Moreover the equations $p_{\ss+-}\Theta\zeta = p_{\ss+-}\Theta^*\zeta$ and
  $p_{\ss-+}\Theta\zeta = p_{\ss-+}\Theta^*\zeta$ are equivalent because
  $p_{\ss+-}\Theta W\zeta = Wp_{\ss-+}\Theta^*\zeta$ and
  $p_{\ss+-}\Theta^* W\zeta = Wp_{\ss-+}\Theta\zeta$.  Finally the equation
  $p_{\ss+-}\Theta\zeta = p_{\ss+-}\Theta^*\zeta$ reads
  $p_{\ss+-}\Theta p_{\ss+-}\zeta +p_{\ss+-}\Theta p_{\ss++}\zeta =
  p_{\ss+-}\Theta^* p_{\ss+-}\zeta + p_{\ss+-}\Theta^* p_{\ss--}\zeta$,
  i.e. $-r\zeta+s\zeta = -r^*\zeta+s'\zeta$, which is equivalent to
  $\Lambda^*\zeta = 0$ since $\zeta = W\zeta$.

  Finally we can compute, using Equations~\eqref{app_lem_eq_+-+-}
  and~\eqref{app_lem_eq_+-+-prime} which read respectively
  $ss^*+rr^* = p_{\ss+-}$ and $s's^{\prime *}+r^*r = p_{\ss+-}$:
  \begin{align*}
    \Lambda^*\Lambda &=
    \Lambda^*p_{\ss++}\Lambda + \Lambda^*p_{\ss+-}\Lambda + 
    \Lambda^*p_{\ss-+}\Lambda + \Lambda^*p_{\ss--}\Lambda \\
    &= 4 ss^* + (r^*-r)(r-r^*) + (r^*-r)(r-r^*) + 4s's^{\prime *} \\ 
    &= 2 (r^*-r)(r-r^*) + 4(\id - rr^*) + 4(\id - r^*r)
    = 8 - 2 (r+r^*)^2.
  \end{align*}
\end{proof}

Recall that $r$ is a direct sum of right shifts with weights
$c_{k,l}\in\itv[0,1]$ converging to $1$ as $k\to\infty$. In particular one sees
that $\|r+r^*\|=2$ so that $0 \in \mathrm{Sp}(\Lambda^*\Lambda)$ and the image
of $\Lambda$ is not closed. Denoting $\Kk$ the ``canonical dense subspace of
$K$'', i.e. the algebraic direct sum of the subspaces $p_n K$, we clearly have
$\Lambda(K_{\ss+-}\cap \Kk) \subset \Kk$ hence $L\cap\Kk$ is a dense subspace of $L$.

\begin{proposition}   \label{prop_descr_Theta}
  There exists an isomorphism $\Upsilon : K_{\ss+-} \to L$ and vectors
  $e_i \in q_1p_1K_{\ss+-}$ such that
  $\Upsilon^* (\Theta+\Theta^*)\Upsilon = -(r+r^*)$ and
  $(h\otimes\Tr)(\Upsilon T \Upsilon^*) = \sum (f_i|Tf_i)$, where
  $f_i = (8-2(r+r^*)^2)^{-1/2}e_i$.
\end{proposition}

\begin{proof}
  We first show that $(\Theta+\Theta^*)\Lambda = -\Lambda(r+r^*)$. Since
  $W\Lambda=\Lambda$ we have
  $p_{\ss++}(\Theta+\Theta^*)\Lambda = 2 p_{\ss++}\Theta\Lambda$ and we compute,
  using the identity~\eqref{app_lem_eq_+++-}:
  \begin{align*}
    p_{\ss++}\Theta\Lambda &= p_{\ss++}\Theta p_{\ss-+} \Lambda
                                  + p_{\ss++}\Theta p_{\ss--}\Lambda
                                  = p_{\ss++}\Theta p_{\ss-+} W (r-r^*) -2 p_{\ss++}\Theta p_{\ss--} s^{\prime *} \\
                                &= s^* (r-r^*) -2 s^*r = -s^*(r+r^*) = -{\ts\frac 12}p_{\ss++}\Lambda(r+r^*).
  \end{align*}
  If we knew that $p_{\ss++}$ is injective on $L$, this would suffice to obtain
  the desired relation because we already know that
  $(\Theta+\Theta^*)(L)\subset L$.  This is true but not completely obvious
  since $\Img \Lambda$ is only dense in $L$. So we check the other
  components. We have, using again~\eqref{app_lem_eq_+-+-} and
  \eqref{app_lem_eq_+-+-prime}:
  \begin{align*}
    p_{\ss+-}\Theta\Lambda &= p_{\ss+-}\Theta p_{\ss+-}\Lambda +
                                  p_{\ss+-}\Theta p_{\ss++}\Lambda
                                  = -r(r-r^*) +2ss^* \text{~~~and} \\
    p_{\ss+-}\Theta^*\Lambda &= p_{\ss+-}\Theta^* p_{\ss+-}\Lambda +
                                    p_{\ss+-}\Theta^* p_{\ss--}\Lambda
                                    = -r^*(r-r^*)-2s's^{\prime *} \text{~~~hence} \\
    p_{\ss+-}(\Theta+\Theta^*)\Lambda &=
                                             (-r^2-rr^*+r^*r+r^{*2}) = -(r-r^*)(r+r^*) = -p_{\ss+-}\Lambda(r+r^*).
  \end{align*}
  Applying $W$ to both sides we obtain
  $p_{\ss-+}(\Theta+\Theta^*)\Lambda = -p_{\ss-+}\Lambda(r+r^*)$. Finally we
  have using~\eqref{app_lem_eq_--+-}:
  \begin{align*}
    p_{\ss--}(\Theta+\Theta^*)\Lambda &= 2 p_{\ss--}\Theta\Lambda
                                        = 2p_{\ss--}\Theta p_{\ss+-}\Lambda + 2 p_{\ss--}\Theta p_{\ss++}\Lambda 
                                        = 2s^{\prime *}(r-r^*) + 4 p_{\ss--}\Theta p_{\ss++} s^* \\
                                      &= 2s^{\prime *}(r-r^*) + 4s^{\prime*}r^* = 2s^{\prime *}(r+r^*) = -p_{\ss--}\Lambda(r+r^*).
  \end{align*}

  Then we perform the polar decomposition of $\Lambda$ as
  $\Lambda = \Upsilon |\Lambda|$, with
  $|\Lambda| = \sqrt{\Lambda^*\Lambda} \in B(K_{\ss+-})$. Since $\Lambda$ has
  dense image in $L$, $\Upsilon \in B(K_{\ss+-},L)$ is a surjective
  isometry. Since $(\Theta+\Theta^*)$, $(r+r^*)$ are self-adjoint, the identity
  $(\Theta+\Theta^*)\Lambda = -\Lambda(r+r^*)$ implies
  $(\Theta+\Theta^*)\Upsilon = -\Upsilon(r+r^*)$.

  To compute $h\otimes\Tr$ we fix an ONB $(\zeta_i)_i$ of $p_1H$, so that we
  have
  $(h\otimes\Tr)(X) = \sum_i (\xi_0\otimes\zeta_i \mid X(\xi_0\otimes\zeta_i))$.
  Observe that $\xi_0\otimes p_1H = p_0 K = q_0p_0K\oplus q_1p_0K$, and we can
  assume that the one-dimensional subspace $q_0p_0K$ is spanned by
  $\xi_0\otimes \zeta_1$. Since $r$, $s$, $s'$, $W$ commute with the projections
  $q_l$, it is also the case for $\Lambda$. In particular the property
  $q_0 K_{\ss+-} = \{0\}$ implies $q_0 L = \{0\}$, hence
  $\xi_0\otimes\zeta_1 \bot L$. On the other hand the vectors
  $\xi_0\otimes \zeta_i$, $i\geq 2$, form a basis of $q_1p_0K$.

  We have then
  $(h\otimes\Tr)(\Upsilon T \Upsilon^*) = \sum_{i\geq 2}
  (\Upsilon^*(\xi_0\otimes\zeta_i)\mid T \Upsilon^*(\xi_0\otimes\zeta_i))$.
  We obtain the formula of the statement by putting
  $f_i = \Upsilon^*(\xi_0\otimes\zeta_i)$. Note moreover that $|\Lambda|$ is
  injective and $|\Lambda| = (8-2(r+r^*)^2)^{1/2}$ by
  Proposition~\ref{prop_descr_L}. Hence $f_i$ has the required form if we define
  $e_i = |\Lambda|(f_i) = \Lambda^*(\xi_0\otimes\zeta_i) = 2
  s^*(\xi_0\otimes\zeta_i) \in p_1 K_{\ss+-}$.
  Since $\xi_0\otimes\zeta_i \in q_1K$ we have $e_i \in q_1p_1 K_{\ss+-}$ as
  claimed.
\end{proof}

\begin{theorem}\label{thm_rev_det_class}
  The element $1+\Re\Theta \in U\Ll(\FO_n)U\otimes B(p_1H)$ is of determinant class with
  respect to the functional $(h\otimes \Tr)$.
\end{theorem}

\begin{proof}
  Denote $p_g^+$, $p_g^-$, $p_a$, $p_L$ the orthogonal projections onto $K_g^+$,
  $K_g^-$, $K_a$ and $L$ respectively. Since they commute with $\Theta+\Theta^*$,
  we have to prove that $(h\otimes\Tr)(\log_+(q(1+\Re\Theta)))$ is finite for
  each projection $q = p_g^+$, $p_g^-$, $p_a$, $p_L$ separately. This is clear
  for $p_g^+$, $p_g^-$ since $1+\Re\Theta = 0$ and $2$ on the corresponding
  subspaces. The term with $p_a$ vanishes since $(h\otimes\Tr)$ is a sum of
  vector states associated to vectors in $p_0K = p_0K_{\ss++}$ which is
  orthogonal to $K_a \subset K_{\ss+-}\oplus K_{\ss-+}$.

  Hence we are left with the term corresponding to $p_L = \Upsilon\Upsilon^*$,
  which according to Proposition~\ref{prop_descr_Theta} is equal to:
  \begin{align*}
    (h\otimes\Tr)(\Upsilon\Upsilon^*\log_+(1+\Re\Theta)\Upsilon\Upsilon^*) &=
    (h\otimes\Tr)(\Upsilon\log_+(1-\Re r)\Upsilon^*) \\
    & = {\ts\sum_i} (f_i \mid \log_+(1-\Re r) f_i) \\
    & = {\ts\frac{1}{8}}{\ts\sum_i} (e_i \mid (1-(\Re r)^2)^{-1}\log_+(1-\Re r) e_i).
  \end{align*}
  We fix $i$ and we put $\eta_0 = e_i/\|e_i\| \in q_1p_1K_{\ss+-}$. According
  to~\ref{app_weights} the map $r$ maps $q_1p_kK_{\ss+-}$ isometrically to
  $q_1p_{k+1}K_{\ss+-}$, up to the scalar $c_{k+1,1} \in \itv]0,1]$, and we have
  $r^*(\eta_0) = 0$. If we define recursively
  $\eta_{k+1} = r\eta_k / \|r\eta_k\|$, this shows that we can identify the
  restriction of $r$ to $\overline{C^*(r)\eta_0}$ with a weighted unilateral
  shift on $\ell^2(\NN) \simeq \overline\Span\{\eta_k\}$. Observe moreover that
  $\eta_0$ lies in the range of $\sqrt{1-(\Re r)^2}$, since
  $e_i = 2\sqrt 2 \sqrt{1-(\Re r)^2} f_i$. The result now follows
  from the following Lemma.
\end{proof}

\begin{lemma}
  Let $R$ be a weighted unilateral shift on $\ell^2(\NN)$ with weights
  $c_k\in\itv]0,1]$ --- in other words $R \delta_k = c_{k+1}\delta_{k+1}$ where
  $(\delta_k)_k$ is the canonical basis of $\ell^2(\NN)$. We assume that
  $\delta_0$ is in the range of $\sqrt{1-(\Re R)^2}$ and we denote $\omega$ the
  vector state associated to $\delta_0$. Then
  $\omega((1-(\Re R)^2)^{-1}\log_+(1-\Re R))$ is finite.
\end{lemma}

\begin{remark}
  Denote by $\mu$ the spectral measure of $\Re(R)$ with respect to $\omega$,
  which is supported on $\itv[{-1},1]$.  Then we have
  $\omega(f(\Re R)) = (\delta_0|f(\Re R)\delta_0) = \int_{-1}^1 f(t)\d\mu(t)$
  for any $f \in L^\infty(\itv[{-1},1])$, and if $f : \itv[{-1},1] \to \RR$ is
  any Borel map we say that $\omega(f(\Re R))$ is finite if $f$ is integrable
  with respect to $\mu$. In the Lemma above we take
  $f(t) = (1-t^2)^{-1}\log_+(1-t)$ and the finiteness of $\omega(f(\Re R))$ is
  equivalent to the convergence, at $1$ and $-1$, of the integral
  \begin{equation}\label{eq_integral}
    \int_{-1}^1 \frac{\log_+(1-t)}{1-t^2} \d\mu(t).
  \end{equation}
\end{remark}
 

\begin{proof}
This kind of result is perhaps well-known to experts in operator theory. However
  we provide an elementary proof for the convenience of the reader.

  We proceed by comparison with the standard unilateral shift
  $R_0 : \delta_k \to \delta_{k+1}$. Recall that the moments
  $m_k(\Re(R_0)) = \omega((\Re R_0)^k)$ are given in terms of the Catalan
  numbers $C_k = \frac{1}{k+1}{2k \choose k}$ by $m_{2k+1} = 0$,
  $m_{2k} = 4^{-k}C_k$ \cite[Corollary 2.14]{MR2266879}. Recall also that the
  Catalan numbers are counting the number of Dyck paths $\pi \in D_{k}$ of
  length $2k$, as can be seen by expanding $(R_0+R_0^*)^{2k}\delta_0$ and
  looking for the $\delta_0$ component.  See \cite[Propositions 2.11 and
  2.13]{MR2266879}. In the case of a general $R$, we still have $m_{2k+1} = 0$
  because $R$ is odd with respect to the natural $\ZZ_2$-grading. Moreover,
  still by expanding $(R+R^*)^{2k}\delta_0$ one sees that the even moments
  $m_{2k}(\Re R)$ are given by a sum over Dyck paths,
  $m_{2k}(\Re R) = 4^{-k} \sum_{\pi\in D_k} c_\pi$, where the contributions
  $c_\pi$ are products of weights $c_k$. In particular we have
  $c_\pi\in\itv]0,1]$ and
  $0\leq m_{2k}(\Re R) \leq 4^{-k}\# D_k = m_{2k}(\Re R_0)$.

  As above, denote by $\mu$, $\mu_0$ the spectral measures of $\Re(R)$ and
  $\Re(R_0)$ with respect to $\omega$, which are both supported on
  $\itv[{-1},1]$.  Note that $f : t\mapsto 1/(1-t^2)$ is $\mu$-integrable
  because $\delta_0$ lies in the range of $\sqrt{1-(\Re R)^2}$: indeed,
  approximating $f$ by $f_C : t \mapsto \min(f(t),C)$ and writing
  $\delta_0 = g(\Re R)\zeta$ with $\zeta\in\ell^2(\NN)$,
  $g : t\mapsto \sqrt{1-t^2}$, we have $|f_C(t)g(t)^2| \leq 1$ hence
  $\int_{-1}^1 f_C(t)\d\mu(t) = (\zeta | (f_Cg^2)(\Re R)\zeta) \leq \|\zeta\|^2$
  for all $C$ and $\int_{-1}^1 f(t)\d\mu(t) \leq \|\zeta\|^2$ by monotone
  convergence.

  In particular the integral~\eqref{eq_integral} converges {\bf iff} the
  corresponding integral over $\itv[0,1]$ is finite. Adding the finite quantity
  $\int_{0}^1 \log_+(1+t)/(1-t^2) \d\mu(t)$ to this new integral, we
  conclude that the convergence of~\eqref{eq_integral} is equivalent to
  \begin{displaymath}
    \int_{0}^1 \frac{\log_+(1-t^2)}{1-t^2} \d\mu(t) = \frac{1}{2}\int_{-1}^1 \frac{\log_+(1-t^2)}{1-t^2} \d\mu(t) > -
    \infty,
  \end{displaymath}
  where in the right-hand integral we have switched back to integrating over
  $[-1,1]$ using the fact that $\mu$ is symmetric.  

  We then perform the power series expansion
  $\log_+(1-t^2)/(1-t^2) = \sum a_k t^{2k}$ on $\itv]{-1},1[$: the convergence
  of~\eqref{eq_integral} is equivalent to the finiteness of
  \begin{displaymath}
    \int_{-1}^1 \frac{\log_+(1-t^2)}{1-t^2} \d\mu(t) =
    \int_{-1}^1 \sum_{k\in\NN} a_k t^{2k} \d\mu(t).
  \end{displaymath}
  Since it is readily seen that all coefficients $a_k$ are non-positive and $t^{2k}$ is non-negative,
  one can permute the sum and the integral and compare to $R_0$:
  \begin{displaymath}
    \int_{-1}^1 \frac{\log_+(1-t^2)}{1-t^2} \d\mu(t) = \sum_{k\in\NN} a_k 
    m_{2k}(\Re R) \geq \sum_{k\in\NN} a_k 
    m_{2k}(\Re R_0) = \int_{-1}^1 \frac{\log_+(1-t^2)}{1-t^2} \d\mu_0(t).
  \end{displaymath}

  Now we can conclude because the spectral measure of $R_0$ with respect to
  $\omega$ is well-known: it is the semicircular law
  $\d\mu_0(t) = \frac 1{\pi} \sqrt{1-t^2}\d t$ \cite[Proposition 2.15]{MR2266879}. Hence we are led to the
  following Bertrand integral, which is well-known to be finite:
  \begin{displaymath}
    \int_{-1}^1 \frac{\log(1-t^2)}{\sqrt{1-t^2}} \d t > -\infty.
  \end{displaymath}
\end{proof}

\section{Free entropy and relations in $\FO_N$}
\label{sec_free_entropy}

In this section we will apply Jung and Shlyakhtenko's
Theorem~\ref{thm_jung_shlyakhtenko} to $M=\Ll(\FO_N) \subset B(H)$. We fix the
tuple of standard generators $u = (u_{ij})_{ij}$, which we now consider as
elements of the {\em reduced} \Cst algebra $C^*_\red(\FO_N)$. We consider the
``canonical'' vector of relations
$F = (F_1,F_2) \in \CC\langle x_{ij}\rangle \otimes (M_N(\CC)\oplus M_N(\CC))$
given by $F_1 = x^tx-1$ and $F_2 = xx^t-1$, with
$x = (x_{kl})_{kl} \in \CC\langle x_{ij}\rangle\otimes M_N(\CC)$. Note that we
have $m=N^2$ and $l=2N^2$ with the notation of Section~\ref{sec_notation}.

Recall from Section~\ref{sec_notation} that $\rk \partial F(u)$ is the
Murray-von Neumann dimension of $\overline{\Im}~ \partial F(u)$ in the right
$M\otimes M^\circ$-module $H\otimes H^\circ\otimes (M_N(\CC)\oplus M_N(\CC))$.
The following Lemma is a straightforward adaptation of 
\cite[Lemma~3.1]{Shlyakhtenko_OneBounded} and its proof, and relies heavily on the computation
of the first $L^2$-Betti number of $\FO_N$ in \cite{Vergnioux_Paths}.

\begin{lemma}[{\cite[Lemma~3.1]{Shlyakhtenko_OneBounded}}]
  \label{lem_rank}
  We have $\rk \partial F(u) = N^2-1$.
\end{lemma}

\begin{proof}
  By definition,  $\CC[\FO_N]$ is the quotient of
  $\CC\langle x_{ij}\rangle$ by the ideal generated by the polynomials
  $F_{pkl}$, $p=1, 2$, $k$, $l = 1,\ldots N$. Recall that $H\otimes H^\circ$ is
  equipped with the $M,M$-bimodule structure corresponding to the left action of
  $M$ (resp. $M^\circ$) on itself. We make it into a
  $\CC\langle x_{ij}\rangle,\CC\langle x_{ij}\rangle$-bimodule by evaluating
  polynomials at $x_{ij}=u_{ij}$, so that
  $P\cdot(\zeta\otimes\xi)\cdot Q = P(u)\zeta \otimes \xi Q(u)$.
  
  A derivation $\delta : \CC\langle x_{ij}\rangle \to H\otimes H^\circ$ factors
  through $\CC[\FO_N]$ {\bf iff} we have $\delta(F_{pkl}) = 0$ for all $p$, $k$,
  $l$ --- indeed by Leibniz' rule and the fact that $F_{pkl}(u) = 0$ this
  implies $\delta(PF_{pkl}Q) = 0$ for all $P$, $Q \in \CC\langle x_{ij}\rangle$.
  Now derivations $\delta : \CC\langle x_{ij}\rangle \to H\otimes H^\circ$ are
  in $1$-$1$ correspondence with the tuples of values
  $\zeta_{ij} = \delta(x_{ij}) \in H\otimes H^\circ$, the derivation
  corresponding to $(\zeta_{ij})_{ij}$ being given by
  $\delta(P) = \partial P\#\zeta = \sum\partial_{ij} P\#\zeta_{ij}$. Then
  $\delta$ factors through $\CC[\FO_N]$ {\bf iff}
  $\partial F\#\zeta = (\partial F_{pkl}\#\zeta)_{pkl} = 0$. Here we use the
  notation $(R\otimes S)\#\xi = R\cdot\xi\cdot S$.

  This shows that the space of derivations
  $\mathrm{Der}(\CC[\FO_N],H\otimes H^\circ)$ is isomorphic as a right
  $M\otimes M^\circ$-module to
  $\Ker\partial F(u) \subset (H\otimes H^\circ)^{m}$, where $m=N^2$. Taking von
  Neumann dimensions, we obtain
  \begin{displaymath}
    \rk\partial F(u) = N^2 - 
    {\ts\dim}_{M\otimes M^\circ} \mathrm{Der}(\CC[\FO_N],H\otimes H^\circ).
  \end{displaymath}
  On the other hand there are general exact sequences for Hochschild cohomology:
\begin{gather*}
  0 \to H^0(\CC[\FO_N],H\otimes H^\circ) \to 
  H\otimes H^\circ \to
  \mathrm{Inn}(\CC[\FO_N],H\otimes H^\circ) \to 0, \\
  0 \to \mathrm{Inn}(\CC[\FO_N],H\otimes H^\circ) \to 
  \mathrm{Der}(\CC[\FO_N],H\otimes H^\circ) \to
  H^1(\CC[\FO_N],H\otimes H^\circ) \to 0.
\end{gather*}
If one uses the definition
$\beta_k^{(2)}(\FO_N) = \dim_{M\otimes M^\circ} H^k(\CC[\FO_N],H\otimes
H^\circ)$ for the $L^2$-Betti numbers (see \cite{Kyed_L2Homology, MR2399103}),
 the additivity of L\"uck-von Neumann dimension readily yields
\begin{displaymath}
  {\ts\dim_{M\otimes M^\circ}} \mathrm{Der}(\CC[\FO_N],H\otimes H^\circ) = 
  \beta_1^{(2)}(\FO_N)-\beta_0^{(2)}(\FO_N)+1.
\end{displaymath}
By \cite[Corollary~5.3]{Vergnioux_Paths} this is equal to $1$, which concludes
the proof.
\end{proof}


Recall that $\partial F(u) = \partial(F_1,F_2)(u)$ is an operator in
$B(H)\otimes B(H)\otimes B(M_N(\CC),M_N(\CC)\oplus M_N(\CC))$. In the next Lemma
we identify $M_N(\CC)$ with $p_1H$ and we make the connection with the reversing
operator studied in Section~\ref{sec_reversing}.

\begin{lemma}\label{lemma_diff_reversing}
  We have
  $\partial(F_1,F_2)(u)^* \partial(F_1,F_2)(u) = 2 \partial F_1(u)^*\partial
  F_1(u) \in B(H\otimes H\otimes M_N(\CC))$
  and $\partial F_1(u)^*\partial F_1(u)$ is unitarily conjugated to
  $2+2\Re(\Theta\otimes 1)\in B(H\otimes p_1H\otimes H)$. Moreover the state
  $(h\otimes h\otimes\Tr)$ is transformed into
  $(h\otimes\Tr\otimes h)(V_{23}^* \,\cdot\, V_{23})$ under the same conjugation.
\end{lemma}

\begin{proof}
  We first compute the free derivatives. We have
  $F_{1kl} = \sum_p x_{pk}x_{pl} - \delta_{kl}$ hence
  $\partial_{ij}F_{1kl} = \delta_{kj} (1\otimes x_{il}) + \delta_{lj}
  (x_{ik}\otimes 1)$.
  Using the matrix units $E_{ij}$ as a basis of $M_N(\CC)$, this yields
  \begin{align*}
    \partial F_1(E_{ij}) &= {\ts\sum_{kl}} \delta_{jk}(1\otimes x_{il}\otimes E_{kl})
    + {\ts\sum_{kl}} \delta_{jl} (x_{ik}\otimes 1\otimes E_{kl}) \\
    &= {\ts\sum_l} (1\otimes x_{il}\otimes E_{jl}) + 
    {\ts\sum_k} (x_{ik}\otimes 1\otimes E_{kj})
  \end{align*}
  so that
  $\partial F_1 = \sum_{il} (1\otimes x_{il}\otimes T\lambda(E_{li})) +
  \sum_{ik} (x_{ik}\otimes 1\otimes \lambda(E_{ki}))$
  where $\lambda(M) \in B(M_N(\CC))$ is the map of left multiplication by $M$,
  and $T \in B(M_N(\CC))$ is the transpose map. Similarly for
  $F_2 = (\sum x_{kp}x_{lp}-\delta_{kl})_{kl}$ we have
  \begin{align*}
    \partial F_2(E_{ij}) &= {\ts\sum_{kl}} \delta_{ik}(1\otimes x_{lj}\otimes E_{kl})
    + {\ts\sum_{kl}} \delta_{il} (x_{kj}\otimes 1\otimes E_{kl}) \\
    &= {\ts\sum_l} (1\otimes x_{lj}\otimes E_{il}) + 
    {\ts\sum_k} (x_{kj}\otimes 1\otimes E_{ki})
  \end{align*}
  and
  $\partial F_2 = \sum_l (1\otimes x_{lj}\otimes T\lambda(E_{lj})T) + \sum_k
  (x_{kj}\otimes 1\otimes \lambda(E_{kj})T)$.

  Then we evaluate at $x_{ij}=u_{ij}$. Recall that
  $C^*_\red(\FO_N)\otimes C^*_\red(\FO_N)^\circ$ acts on $H\otimes H$ by
  $\id\otimes\rho$, where $\rho(x)$ is the map of right multiplication by $x$,
  which can also be written $\rho(x) = US(x)U$ in the Kac case. Here $S$ is
  the antipode and we have in particular $S(u_{ij}) = u_{ji}$. We obtain in
  $B(H\otimes H\otimes M_N(\CC))$:
  \begin{align*}
    \partial F_1(u) &= \sum (1\otimes U u_{li}U\otimes T\lambda(E_{li})) +
                      \sum(u_{ik}\otimes 1\otimes \lambda(E_{ik}^*)) \\
                    &= (1\otimes U\otimes T)(\id\otimes\lambda)(u_{32})
                      (1\otimes U\otimes 1) + (\id\otimes\lambda)(u_{31}^*),    
  \end{align*}
  where $u \in M_N(C^*_\red(\FO_N)) \simeq M_N(\CC) \otimes C^*_\red(\FO_N)$. 

  Now we identify $M_N(\CC) = B(\CC^N)$ with $p_1H$ using the decomposition of
  the multiplicative unitary recalled in Section~\ref{sec_notation} --- in
  particular we have then
  $(\id\otimes\lambda)(u_{21}) = u_{21} \in C^*_\red(\Gamma)\otimes p_1c_0(\Gamma)
  \subset B(H\otimes p_1H)$.
  Moreover in this identification $T$ corresponds with the restriction of $U$ to
  $p_1H$. Hence we have finally
  $\partial F_1(u) = (1\otimes U\otimes U)u_{32}(1\otimes U\otimes 1) +
  u_{31}^*$,
  which is also the restriction of
  $(1\otimes U\otimes U)V_{32}(1\otimes U\otimes 1) + V_{31}^*$ to
  $H\otimes H\otimes p_1H$. This is a sum of two unitaries and we obtain in
  particular $\partial F_1(u)^*\partial F_1(u) = 2 + 2\Re W$ on
  $H\otimes H\otimes p_1H$, where
  $W = V_{31}(1\otimes U\otimes U)V_{32}(1\otimes U\otimes 1)$.

  Proceeding similarly with $F_2$ we obtain
  \begin{align*}
    \partial F_2(u) &= \sum (1\otimes Uu_{jl}U\otimes T\lambda(E_{lj})T) + 
                      \sum (u_{kj}\otimes 1\otimes \lambda(E_{kj})T) \\
    &= (1\otimes U\otimes T)(\id\otimes\lambda)(u_{32}^*)(1\otimes U\otimes T) +
      (\id\otimes\lambda)(u_{31})(1\otimes 1\otimes T) \\
    &= (1\otimes U\otimes U) V_{32}^* (1\otimes U\otimes U) +
      V_{31}(1\otimes 1\otimes U) \text{~~~on $H\otimes H\otimes p_1H$,}
  \end{align*}
  and
  $\partial F_2(u)^*\partial F_2(u) = 2 + 2 \Re(1\otimes U\otimes U)V_{32}
  (1\otimes U\otimes U)V_{31}(1\otimes 1\otimes U)$.
  We moreover observe that
  $(1\otimes U\otimes U)V_{32} (1\otimes U\otimes U) \in 1\otimes B(H) \otimes
  Uc_0(\FO_N)U$
  and $V_{31} \in B(H)\otimes 1\otimes c_0(\FO_N)$. Since
  $[c_0(\Gamma), Uc_0(\Gamma)U] = 0$, we can permute both terms and we obtain
  $\partial F_2(u)^*\partial F_2(u) = 2 + 2 \Re W = \partial F_1(u)^*\partial
  F_1(u)$.

  Now we perform unitary conjugations to ``simplify'' $W$. We first use
  $U_2\Sigma_{23}$ which yields the symmetric form
  $W\sim_u V_{21}U_2 V_{23} \in B(H\otimes p_1 H\otimes H)$. Conjugating further
  by $U_1$ we obtain $W \sim_u \tilde V_{12}U_2V_{23}$, where
  $\tilde V = \Sigma(1\otimes U)V(1\otimes U)\Sigma$. Finally we conjugate by
  $V_{13}V_{23}$ and we use the formula
  $V_{13}V_{23}\tilde V_{12} = \tilde V_{12}V_{13}$ from
  \cite[Proposition~6.1]{BaajSkandalis}:
  \begin{displaymath}
    W \sim_u V_{13}V_{23}\tilde V_{12}U_2V_{13}^* = 
    \tilde V_{12}V_{13}U_2V_{13}^* = \tilde V_{12} U_2 = \Theta\otimes 1.
  \end{displaymath}

  Notice at last that $h\otimes h\otimes\Tr$ is a sum of vector states
  associated to vectors of the form $\xi_0\otimes\xi_0\otimes\zeta$. We have
  $U\xi_0 = \xi_0$ and $V(\xi_0\otimes 1) = \xi_0\otimes 1$. Applying the
  various unitaries used to transform $W$ we obtain
  \begin{displaymath}
    V_{13}V_{23}U_1\Sigma_{23}U_2 (\xi_0\otimes\xi_0\otimes\zeta) = 
    V_{13}V_{23} (\xi_0\otimes\zeta\otimes\xi_0) = V_{23}
    (\xi_0\otimes\zeta\otimes\xi_0)
  \end{displaymath}
  and the last claim follows.
\end{proof}

Thanks to Theorem~\ref{thm_rev_det_class} we can finally prove our main theorem:
\begin{theorem}\label{s1bdd}
  The von Neumann algebra $\Ll(\FO_N)$ is strongly $1$-bounded for all
  $N \ge 3$.
\end{theorem}

\begin{proof}
  The $1$-boundedness of the tuple of generators $u = (u_{ij})$ of $\Ll(\FO_N)$
  is a straightforward consequence of Jung's and Shlyakhtenko's
  Theorem~\ref{thm_jung_shlyakhtenko}, applied to $u$ and to the relations
  $F = (F_{pkl})$ introduced at the beginning of this section. Note that on the
  matrix algebra $B(p_1H)\otimes 1$, any positive functional, and in particular
  $(\Tr\otimes h)(V^* \cdot V)$, is dominated by a multiple of the standard
  trace $\Tr\otimes h$. Then $\partial F(u)$ is of determinant class with
  respect to $(h\otimes h\otimes\Tr)$ by Lemma~\ref{lemma_diff_reversing} and
  Theorem~\ref{thm_rev_det_class}. Moreover $N^2-\rk \partial F(u) = 1$ by
  Lemma~\ref{lem_rank}.

  Strong $1$-boundedness of $\Ll(\FO_N)$ will now follow if at least one of the
  generators $u_{ij}$ has finite free entropy. Recall that for a single
  self-adjoint variable $X = X_1$ in a finite von Neumann algebra $(M,\tau)$ with
  law $\mu$, we have the formula \cite[Proposition~4.5]{MR1296352}
  \begin{displaymath}
    \chi(X) = \iint \log |s-t| \d\mu(s)\d\mu(t) + C.
  \end{displaymath}
  In particular if $\mu$ has an essentially bounded density with respect to the
  Lebesgue measure (and is compactly supported), then $\chi(X)$ is evidently
  finite. This is indeed the case for all generators $u_{ij}$ of $\Ll(\FO_N)$ ---
  according to \cite[Theorem~5.3]{BanicaCollinsZinnJustin} the density is even
  continuous.
\end{proof}

\begin{corollary} \label{consequence} For $N\geq 3$ the von Neumann algebra
  $\Ll(\FO_N)$ is not isomorphic to any finite von Neumann algebra (with
  separable predual) which admits a tuple of self-adjoint generators $X$ with
  $\delta_0(X)>1$. In particular it is not isomorphic to any free group factor
  $\Ll(F_n)$.
\end{corollary}

More generally, $\Ll(\FO_N)$ is not isomorphic to any interpolated free group
factor $\Ll(F_r)$, nor to any group von Neumann algebra $\Ll(\Gamma)$ where
$\Gamma = {*}_{i=1}^k \ZZ/n_i\ZZ$ is a free product of cyclic groups, for
instance $\Gamma = (\ZZ/2\ZZ)^{*N}$. Indeed these von Neumann algebras admit
tuples of self-adjoint generators with $\delta_0 = r$,
$\delta_0 = k-\sum_{i=1}^k n_i^{-1}$ respectively and these values are strictly
bigger than $1$ in the non amenable cases. According to
\cite[Lemma~3.7]{Jung_OneBounded}, $\Ll(\FO_N)$ is not isomorphic either to any
free product of $R^\omega$-embeddable diffuse finite von Neumann algebras.

\appendix
\section{Computation rules in quantum Cayley trees}
\label{sec_appendix}
\renewcommand{\theequation}{\Alph{section}.\arabic{equation}}

In this appendix we recall definitions and results from
\cite{Vergnioux_Cayley,Vergnioux_Paths} about quantum Cayley graphs for discrete
quantum groups. We use the notation about discrete quantum groups recalled in
Section~\ref{sec_notation}.

\begin{definition}[{\cite[Definition~3.1]{Vergnioux_Cayley}}]
  Let $\Gamma$ be a discrete quantum group, and fix a central projection
  $p_1 \in Z(M(c_0(\Gamma)))$ such that $Up_1 = p_1U$ and $p_0p_1 = 0$. The {\em
    quantum Cayley graph} $X$ \cite{Vergnioux_Cayley} associated to
  $(\Gamma, p_1)$ is given by
\begin{itemize}
\item the vertex and edge Hilbert spaces $\ell^2(X^{(0)}) = \ell^2(\Gamma)$ and
  $\ell^2(X^{(1)}) = \ell^2(\Gamma)\otimes p_1\ell^2(\Gamma)$,
\item the vertex and edge \Cst algebras $c_0(X^{(0)}) = c_0(\Gamma)$ and
  $c_0(X^{(1)}) = c_0(\Gamma)\otimes p_1c_0(\Gamma)$, naturally represented on
  the corresponding Hilbert spaces,
\item the antilinear involutions $J^{(0)} = J$ and $J^{(1)} = J\otimes J$ on
  $\ell^2(X^{(0)})$ and $\ell^2(X^{(1)})$,
\item the boundary operator
  $E = V \in B(\ell^2(X^{(1)}), \ell^2(X^{(0)})\otimes \ell^2(X^{(0)}))$.
\item the reversing operator
  $\Theta = \Sigma(1\otimes U)V(U\otimes U)\Sigma \in B(\ell^2(X^{(1)}))$,
\end{itemize}
We denote $\ell^2(X^{(0)}) = \ell^2(\Gamma) = H$ and
$\ell^2(X^{(1)}) = H\otimes p_1H = K$. We also consider the source and target
operators $E_1 = (\id\otimes\epsilon) E$,
$E_2 : (\epsilon\otimes\id)E : K \to H$, which are {\em a priori} only densely
defined. 
\end{definition}

Recall that $H$ is the GNS space for the Haar state $h \in C^*(\Gamma)^*$, with
canonical cyclic vector $\xi_0$. Denote $\CC[\Gamma]$ the canonical dense
Hopf $\ast$-subalgebra of $C^*(\Gamma)$. The following Proposition computes the
structure maps of the quantum Cayley graph in terms of the Hopf algebra
structure of $\CC[\Gamma]$.

\begin{proposition}[{\cite[Lemma~3.5, Proposition~3.6]{Vergnioux_Cayley}}]
  \label{app_reverse_target}
  For any $x$, $y\in\CC[\Gamma]$ we have
  \begin{itemize}
  \item $\Theta (x\xi_0\otimes y\xi_0) = (\id\otimes S)((x\otimes
    1)\Delta(y))(\xi_0\otimes\xi_0)$,
  \item $E_1(x\xi_0\otimes y\xi_0) = \epsilon(y)x\xi_0$ and
    $E_2(x\xi_0\otimes y\xi_0) = xy\xi_0$.
  \end{itemize}
  Moreover $E_1\Theta = E_2$ and $E_2\Theta = E_1$. If $p_1 \in c_0(\Gamma)$,
  then $E_1$ and $E_2$ are bounded.
\end{proposition}

A key feature of the quantum case is that the reversing operator need not be
involutive --- in fact if $p_1$ is ``generating'' $\Theta$ is involutive {\bf
  iff} $\Gamma$ is a genuine discrete group
\cite[Proposition~3.4]{Vergnioux_Cayley}. Hence the study of the following
eigenspaces becomes non trivial, important, and turns out to be useful for
applications:

\begin{definition}[{\cite[Definition~3.1]{Vergnioux_Cayley}}]
  The space of antisymmetric (or geometric) edges is $K_g^+ =
  \Ker(\Theta+\id)$. The space of symmetric edges is $K_g^- = \Ker(\Theta-\id)$.
\end{definition}

Recall the isomorphism
$c_0(\Gamma) = c_0{-}\bigoplus_{\alpha\in\Irr(\Gamma)} p_\alpha c_0(\Gamma)$
with $p_\alpha c_0(\Gamma) \simeq B(H_\alpha)$. The classical Cayley graph
introduced in the next Definition is an essential tool for the study of the
quantum Cayley graph.

\begin{definition}[{\cite[Definition~3.1, Lemma~4.4]{Vergnioux_Cayley}}]
  Denote $\Dd \subset \Irr(\Gamma)$ the subset such that
  $p_1 = \sum_{\alpha\in\Dd} p_\alpha$. The {\em classical Cayley graph} $G$
  associated to $(\Gamma,p_1)$ is given by
  \begin{itemize}
  \item the vertex set $G^{(0)} = \Irr(\Gamma)$,
  \item the edge set
    $G^{(1)} = \{(\alpha,\beta,\gamma,i) ~;~ \alpha,\beta\in\Irr(\Gamma),
    \gamma\in\Dd, 1\leq i\leq\dim\Hom(\beta,\alpha\otimes\gamma)\}$,
  \item the boundary map
    $e : G^{(1)} \to G^{(0)}\times G^{(0)}, (\alpha,\beta,\gamma,i) \mapsto
    (\alpha,\beta)$,
  \item the reversing map
    $\theta : G^{(1)} \to G^{(1)}, (\alpha,\beta,\gamma,i) \mapsto
    (\beta,\alpha,\bar\gamma,i)$.
  \end{itemize}
  The component $\gamma$ of an edge is called its {\em direction}. If $e$ is
  injective, and in particular if the classical Cayley graph $G$ is a tree,
  $G^{(1)}$ can and will be identified with
  $\{(\alpha,\beta) \in \Irr(\Gamma)^2 \mid \exists \gamma\in\Dd ~
  \beta\subset\alpha\otimes\gamma\}$.
  The {\em origin} of the classical Cayley graph $G$ is the trivial
  corepresentation $\tau$ and if $G$ is connected we denote
  $|\alpha| = d(\tau,\alpha)$ the distance of a vertex $\alpha$ to the origin.
\end{definition}

To any subset $A \subset G^{(0)}$ one can associate the projection
$p = \sum_{\alpha\in A} p_\alpha \in M(c_0(X^{(0)}))$.  One can
proceed similarly with $G^{(1)}$ and $c_0(X^{(1)})$, but one can also associate
to $B \subset G^{(0)}\times G^{(0)}$ the projection
$p = \sum_{(\alpha,\beta)\in B} E^*(p_\alpha\otimes p_\beta)E \in M(c_0(X^{(1)}))$, whose
image corresponds to the space of edges going from $\alpha$ to $\beta$. This
motivates the following definition.

\begin{definition}
  Assume that the classical Cayley graph $G$ associated to $(\Gamma,p_1)$ is a
  tree. For any $n \in\NN$ one puts $p_n =\sum_{|\alpha|=n} p_\alpha \in B(H)$
  and $p_n = \sum_{|\alpha|=n} p_\alpha\otimes p_1 \in B(K)$. The left and right
  projections onto ascending edges are
  $p_{\ss\jok+} = \sum_n E^*(p_n\otimes p_{n+1})E \in B(K)$ and
  $p_{\ss+\jok} = J^{(1)}p_{\ss\jok+}J^{(1)} \in B(K)$. The projections onto
  descending edges are $p_{\ss\jok-} = \sum_n E^*(p_n\otimes p_{n-1})E$ and
  $p_{\ss-\jok} = J^{(1)}p_{\ss\jok-}J^{(1)}$. Finally one puts
  $p_{\ss++} = p_{\ss+\jok} p_{\ss\jok+}$,
  $p_{\ss+-} = p_{\ss+\jok} p_{\ss\jok-}$,
  $p_{\ss-+} = p_{\ss-\jok} p_{\ss\jok+}$,
  $p_{\ss--} = p_{\ss-\jok} p_{\ss\jok-}$ and $K_{\ss++} = p_{\ss++} K$,
  $K_{\ss+-} = p_{\ss+-} K$, $K_{\ss-+} = p_{\ss-+} K$,
  $K_{\ss--} = p_{\ss--} K$.
\end{definition}

In the classical case one has $p_{\ss+\jok} = p_{\ss\jok+}$ so that
$p_{\ss+-} = p_{\ss-+} = 0$. However in the quantum case $p_{\ss+-}$ and
$p_{\ss-+}$ are in general non-zero, and this is related to the non involutivity
of the reversing operator. The following proposition shows indeed that
$p_{\ss+-}\Theta p_{\ss+-}$ acts as a right shift in the decomposition
$K_{\ss+-} = \bigoplus p_n K_{\ss+-}$.

\begin{proposition}[{\cite[Proposition~4.3]{Vergnioux_Cayley}}]\label{app_prop_orient}
  Assume that the classical Cayley graph associated to $(\Gamma,p_1)$ is a
  tree. Then we have
  \begin{itemize}
  \item $p_{\ss\jok+}+p_{\ss\jok-} = \id$ and $p_{\ss+\jok}+p_{\ss-\jok} = \id$,
    $[p_{\ss\jok\pm},p_{\ss\pm\jok}] = 0$,
    $[p_{\ss\pm\jok},p_n] = [p_{\ss\jok\pm},p_n] = 0$,
  \item $\Theta p_{\ss+\jok} p_n = p_{n+1}p_{\ss\jok-}\Theta$ and
    $\Theta p_{\ss-\jok} p_n = p_{n-1}p_{\ss\jok+}\Theta$,
   $p_{\ss\jok -} = \Theta p_{\ss+\jok}\Theta^*$ and
    $p_{\ss-\jok} = \Theta^*p_{\ss\jok+}\Theta$,
  \item $E_2 p_{\ss+-} = E_2 p_{\ss-+} = 0$ and
    $p_n E_2 = E_2 p_{n-1} p_{\ss++} + E_2 p_{n+1}p_{\ss--}$.
  \end{itemize}
  We denote $r = - p_{\ss+-}\Theta p_{\ss+-}$ and
  $s = p_{\ss+-}\Theta p_{\ss++}$.
\end{proposition}

In the rest of the Appendix we consider the case of a free product of orthogonal
and unitary free quantum groups:
$\Gamma = F = \FO(Q_1)*\cdots*\FO(Q_k)*\FU(R_1)*\cdots*\FU(R_l)$, endowed with
the projection $p_1 = \sum_{\alpha\in\Dd} p_\alpha$ associated to the set $\Dd$
of fundamental representations of the factors $\FO(Q_i)$, $\FU(R_j)$, together
with their duals. This is a little bit stronger than requiring the classical
Cayley graph $G$ to be a tree, see \cite[Proposition~4.5]{Vergnioux_Cayley}. In
that case we have the following useful facts. Note that the identities
$(p_{\ss++}+p_{\ss--})\Theta^n(p_{\ss++}+p_{\ss--}) =
(p_{\ss++}+p_{\ss--})\Theta^{-n}(p_{\ss++}+p_{\ss--})$
can be interpreted as a weak involutivity property.

\begin{proposition}[{\cite[Proposition~4.7, Proposition~5.1]{Vergnioux_Cayley}}]\label{app_prop_special}
  Consider the quantum Cayley graph of $(F,p_1)$. Then the restriction
  $E_2 : K_{\ss++} \to H$ is injective and we have
  $(p_{\ss++}+p_{\ss--})\Theta^n(p_{\ss++}+p_{\ss--}) =
  (p_{\ss++}+p_{\ss--})\Theta^{-n}(p_{\ss++}+p_{\ss--})$
  for all $n$. 
\end{proposition}


Here is an example of computation using the rules above:

\begin{lemma}
  In the quantum Cayley graph of $(F,p_1)$ we have 
  \begin{align} \label{app_lem_eq_--+-}
    p_{\ss--}\Theta p_{\ss++}\Theta^* p_{\ss+-} &=
    - p_{\ss--}\Theta p_{\ss+-}\Theta^* p_{\ss+-} \\
    \label{app_lem_eq_+++-}
    p_{\ss++}\Theta p_{\ss--}\Theta p_{\ss+-} &= 
    - p_{\ss++}\Theta^*p_{\ss+-}\Theta p_{\ss+-} \\
    \label{app_lem_eq_+-+-}
    p_{\ss+-}\Theta p_{\ss++}\Theta^* p_{\ss+-} &= 
    p_{\ss+-} - p_{\ss+-}\Theta p_{\ss+-}\Theta^* p_{\ss+-} \\
    \label{app_lem_eq_+-+-prime}
    p_{\ss+-}\Theta^*p_{\ss--}\Theta p_{\ss+-} &= 
    p_{\ss+-} - p_{\ss+-}\Theta^* p_{\ss+-} \Theta p_{\ss+-}.
  \end{align}
\end{lemma}

\begin{proof}
  For the first identity it suffices to use~\ref{app_prop_orient} and write
  $p_{\ss--}\Theta p_{\ss++}\Theta^* p_{\ss+-} + p_{\ss--}\Theta
  p_{\ss+-}\Theta^* p_{\ss+-} = p_{\ss--}\Theta p_{\ss\jok+}\Theta^* p_{\ss+-} +
  p_{\ss--}\Theta p_{\ss\jok-}\Theta^* p_{\ss+-} = p_{\ss--}\Theta\Theta^*
  p_{\ss+-} = 0$.
  The second identity is proved similarly after replacing
  $p_{\ss++}\Theta p_{\ss--}$ with $p_{\ss++}\Theta^* p_{\ss--}$ on the
  left-hand side thanks to~\ref{app_prop_special}. The third identity appears
  already in the proof of \cite[Proposition~6.2]{Vergnioux_Cayley}. For the last
  one we note that, according to~\ref{app_prop_orient},
  $p_{\ss+-}\Theta^* p_{\ss--} \Theta p_{\ss+-} = p_{\ss+-}\Theta^* p_{\ss-\jok}
  \Theta p_{\ss+-}$
  and
  $p_{\ss+-}\Theta^* p_{\ss+-} \Theta p_{\ss+-} = p_{\ss+-}\Theta^* p_{\ss+\jok}
  \Theta p_{\ss+-}$.
  Adding both operators we obtain
  $p_{\ss+-}\Theta^* \Theta p_{\ss+-} = p_{\ss+-}$.
\end{proof}

The subspaces $K_{\ss+-}$ and $K_{\ss-+}$ are strongly connected to each other
through the {\em reflection operator} $W$ introduced as follows.

\begin{proposition}[{\cite[Lemma~5.2]{Vergnioux_Cayley}}]\label{app_def_W}
  Consider the quantum Cayley graph of $(F,p_1)$. There exists a unique unitary
  operator $w : K_{\ss+-} \to K_{\ss-+}$ such that
  $w (p_{\ss+-}\Theta)^n p_{\ss++} = (p_{\ss-+}\Theta^{-1})^n p_{\ss++}$ for all
  $n\geq 1$. We denote
  $W = w p_{\ss+-} + w^* p_{\ss-+} + p_{\ss++} + p_{\ss--} \in B(K)$: this is an
  involutive unitary operator such that $W\Theta W=\Theta^*$,
  $W p_{\ss++} = p_{\ss++}$, $W p_{\ss--} = p_{\ss--}$ and
  $W p_{\ss+-} = p_{\ss-+} W$.
\end{proposition}

The families of projections $\{p_n\}$,
$\{p_{\ss++},p_{\ss+-},p_{\ss-+},p_{\ss--}\}$ form two commuting partitions of
the unit in $B(K)$. To perform the most precise analysis of $\Theta$ one needs
to further decompose the space $K$. Observe that there are $4$ commuting
representations of $c_0(\Gamma)$ on $K$, given by
$\pi_4 : c_0(\Gamma)^{\otimes 4} \to B(K) = B(H\otimes p_1H), x\otimes y\otimes
y'\otimes x' \mapsto (x\otimes y)(Ux'U\otimes Uy'U)$.
In the case of $\FO_N$, the subspaces $p_n K \simeq B(H_n\otimes H_1)$ are
irreducible with respect to $\pi_4$, and the subspaces
$p_{\ss\pm\pm} p_n K \simeq B(H_{n\pm 1},H_{n\pm 1})$ are irreducible with
respect to the representation $\pi_4 \circ (\Delta\otimes\Delta)$ of
$c_0(\Gamma)\otimes c_0(\Gamma)$. The reversing operator $\Theta$ does not
commute to these representations, but it does commute to $\pi_4 \circ \Delta^3$
\cite[Proposition~3.7]{Vergnioux_Cayley} and we consider:

\begin{definition}
  Denote $q_l = \pi_4\Delta^3(p_{2l}) \in B(K)$. We have $\sum q_l =\id$ and
  $[q_l,p_k] = [q_l,p_{\ss\pm\pm}]=0$. We have $p_{\ss++}q_lp_k \neq 0$ {\bf
    iff} $0\leq l\leq k+1$, $p_{\ss\pm\mp}q_lp_k \neq 0$ {\bf
    iff} $1\leq l\leq k$ and $p_{\ss--}q_lp_k \neq 0$ {\bf
    iff} $0\leq l\leq k-1$.
\end{definition}

The subspace $q_0 K$ is the ``classical subspace'' of the space of edges, see
\cite[Remarks~6.4]{Vergnioux_Cayley} and
\cite[Reminder~4.3]{Vergnioux_Paths}. In particular $\Theta^2 = \id$ on $q_0K$.

\begin{proposition}[{\cite[Lemma~6.3]{Vergnioux_Cayley}}]
  \label{app_weights}
  Consider the quantum Cayley graph of $(\FO_N,p_1)$. For all choices of signs
  the operator $p_{\ss\pm\pm}\Theta p_{\ss\pm\pm}$ is a multiple of an isometry
  on each subspace $q_l p_{\ss\pm\pm} p_n K$ and we have, for $\mu$, $\nu$,
  $\mu'$, $\nu' \in \{+,-\}$ and $p_{\ss\mu'\nu'}p_{k} q_l \neq 0$:
  \begin{align*}
    \|p_{\ss\mu\nu}\Theta p_{\ss\mu'\nu'}p_k q_l\| = \left\{
      \begin{array}{ll}
        c_{k+1,l} & \text{if } \mu\nu=\mu'\nu' \text{ and } \mu'=+ \\
        s_{k+1,l} & \text{if } \mu\nu\neq\mu'\nu' \text{ and } \mu'=+
      \end{array} \right. \\
        \|p_{\ss\mu\nu}\Theta p_{\ss\mu'\nu'}p_{k} q_l\| = \left\{
      \begin{array}{ll}
        c_{k,l} & \text{if } \mu\nu=\mu'\nu' \text{ and } \mu'=- \\
        s_{k,l} & \text{if } \mu\nu\neq\mu'\nu' \text{ and } \mu'=-
      \end{array}
      \right.
  \end{align*}
  where
  $s_{k,l}^2 = \ds\frac{\qdim(\alpha_l) \qdim(\alpha_{l-1})} {\qdim(\alpha_k)
    \qdim(\alpha_{k-1})}$
  and $c_{k,l}^2 + s_{k,l}^2 = 1$, with the convention that
  $\qdim(\alpha_{-1})=0$.
\end{proposition}
Note that this statement corrects \cite[Lemma~6.3]{Vergnioux_Cayley} in the case
when $\mu'=-$, which is not used in that article.

\bibliographystyle{plain} 
\bibliography{1-bounded} 

\bigskip

\author{Michael Brannan}:
\address{Department of Mathematics,
Mailstop 3368, Texas A\&M University, 
College Station, TX 77843-3368, USA.}
\email{mbrannan@math.tamu.edu} \\

\author {Roland Vergnioux}:
\address{Normandie Univ, UNICAEN, CNRS, 
LMNO, 14000 Caen, France.} \\
\email{roland.vergnioux@unicaen.fr}

\end{document}